\documentclass[11pt]{amsart}
\usepackage{amsthm}
\usepackage{amssymb}
\usepackage{graphicx}									 %Include graphics
\usepackage{stmaryrd}                  %Includes \mapsfrom
\usepackage{footmisc}                  %Something for footnotes
\usepackage{paralist}
\usepackage{wrapfig}
\usepackage[draft]{pdfcomment}
\usepackage{bbm}
\usepackage[T2A,T1]{fontenc}
\usepackage[utf8]{inputenc}
\usepackage[arrow, matrix, curve]{xy}
\usepackage{color}

\usepackage{enumitem}

\newtheorem{thm}{Theorem}[section]

\newtheorem{lem}[thm]{Lemma}

\newtheorem{cor}[thm]{Corollary}
\newtheorem{prp}[thm]{Proposition}
\newtheorem{cnj}[thm]{Conjecture}

\newtheorem{maintheorem}{Theorem}

\theoremstyle{definition}
\newtheorem{dfn}[thm]{Definition}

\newtheorem{rem}[thm]{Remark}

\theoremstyle{remark}

\newcommand{\Co}{\mathbb{C}}
\newcommand{\R}{\mathbb{R}}
\newcommand{\RR}{\mathbb{R}}

\newcommand{\Z}{\mathbb{Z}}
\newcommand{\ZZ}{\mathbb{Z}}

\newcommand{\Zp}{\mathbb{Z}_p}

\newcommand{\QQ}{\mathbb{Q}}
\newcommand{\orb}{{\mathrm{orb}}}
\newcommand{\Hom}{{\mathrm{Hom}}}
\newcommand{\ind}{{\mathrm{ind}}}
\newcommand{\Fr}{F}
\newcommand{\Ss}{\mathbb{S}}
\newcommand{\K}{K}

\newcommand{\To}{\rightarrow}

\newcommand{\Orb}{\mathcal{O}}
\newcommand{\Or}{\mathrm{O}}
\newcommand{\SOr}{\mathrm{SO}}
\newcommand{\SUr}{\mathrm{SU}}

\title[Odd-dimensional Besse orbifolds are developable]{Odd-dimensional orbifolds with all geodesics closed are covered by manifolds}

\author[Amann]{Manuel Amann}
\address[Amann]{Institut f\"ur Mathematik, Differentialgeometrie, Universit\"at Augsburg, Universit\"atsstra\ss{}e 14, 86159 Augsburg, Germany}
\email{manuel.amann@math.uni-augsburg.de}

\author[Lange]{Christian Lange}
\address[Lange]{Mathematisches Institut der Universit\"at zu K\"oln, Weyertal 86-90, 50931 K\"oln, Germany}
\email{clange@math.uni-koeln.de}

\author[M. Radeschi]{Marco Radeschi}
\address[Radeschi]{University of Notre Dame, Department of Mathematics, 255 Hurley, Notre Dame, IN 46556 }
\email{mradesch@nd.edu}
\thanks{}

\begin{document}

\maketitle

\begin{abstract}
Manifolds all of whose geodesics are closed have been studied a lot, but there are only few examples known. The situation is different if one allows in addition for orbifold singularities. We show, nevertheless, that the abundance of new examples is restricted to even dimensions. As one key ingredient we provide a characterization of orientable manifolds among orientable orbifolds in terms of characteristic classes.
\end{abstract}

%\tableofcontents

\section{Introduction}

It is an essential question in Riemannian geometry to understand the behaviour of closed geodesics and to relate it to properties of the ambient space.

A vast literature has been devoted to studying the existence, and number of closed geodesics in compact Riemannian manifolds. Often results in this respect, or their proofs, depend on the topology of the space. Striking examples are that every metric on a closed manifold has at least one closed geodesic due to Lyusternik--Fet \cite{LuFe}, and that, by the combined works of Gromoll--Meyer \cite{MR0264551} and Sullivan--Vigu\'e-Poirrier \cite{MR0455028}, every closed Riemannian manifold admits infinitely many closed geodesics as long as its rational cohomology is not generated by one element.

Another, somewhat opposite direction of research, has been to understand the Riemannian manifolds all of whose geodesics are closed \cite{Besse}, here called \emph{Besse manifolds}. The compact rank one symmetric spaces (CROSSes) $\Ss^n$, $\mathbb{CP}^n$, $\mathbb{HP}^n$, $\operatorname{CaP}^2$ are the canonical examples. These are in fact the only simply connected spaces known to support a \emph{Besse metric}. Moreover, by the collective work of Bott, Samelson and McCleary, their cohomology rings are the only ones that can show up in the simply connected setting \cite{MR0073993,Samelson,McCl}:

\begin{thm}[Bott--Samelson--McCleary]\label{BottSam}
Given $M$ a Besse manifold, its universal cover $\tilde M$ is also Besse and
\begin{align*}
H^*(\tilde M;\Z)\cong H^*(X;\Z)
\end{align*}
with $X$ a CROSS as above.

In particular, if $n:=\dim \tilde M$ is odd, then $\tilde M$ is homeomorphic to $\Ss^n$ (using the Poincar\'e Conjecture).
\end{thm}

Apart from this result, very little is known: for example, it is not known if exotic spheres admit Besse metrics. Furthermore, except a couple of constructions giving rise to infinite families of Besse metrics on spheres, no other Besse metrics are known. In particular, it is not known if there are non-canonical Besse metrics on the projective spaces. As a matter of fact, Lin--Schmidt proved that the only Besse metric on a real projective space of dimension at least $4$ is the round one \cite{LiSch}.

The study of closed geodesics can be extended to the \emph{Riemannian orbifold} setting (cf. Section \ref{sub:basics_on_orbifolds}). In particular, it makes sense to search for closed geodesics on Riemannian orbifolds \cite{MR1419471,GuHa,Dragomir} and to investigate \emph{Besse orbifolds}, that is, Riemannian orbifolds all of whose geodesics are closed \cite{MR2529473,La}. Allowing for orbifold singularities gives rise to interesting new phenomena and questions. For instance, while Guruprasad and Haefliger proved the GM--SVP result for Riemannian orbifolds and the Lyusternik-Fet result for non-developable orbifolds \cite{GuHa}, it is in general not known if every closed Riemannian orbifold has a closed geodesic. Moreover, it provides a much wider class of examples: one immediately gets infinitely many examples of Besse orbifolds in every even dimension, given by weighted projective spaces (see Section \ref{sub:basics_on_orbifolds}) with interesting properties; e.g. their geodesic length spectra are much more complicated than in the manifold case \cite{MR2529473}.

With this in mind, it may come as a surprise that in odd dimensions even the orbifold setting is very rigid, i.e. not richer than the manifold category, which is exactly expressed by our main result.
\begin{maintheorem}\label{T:main-theorem}
Any odd-dimensional Besse orbifold is of the form $(M,g)/\Gamma$ for some Riemannian manifold $M$ homeomorphic to a sphere, a Besse metric $g$, and some finite group $\Gamma$ of isometries of $(M,g)$. Moreover, if $n\neq 3$ and $\Gamma$ contains a fixed-point-free involution, then $(M,g)$ is isometric to a round sphere and $\Gamma\subset \Or(n+1)$.
\end{maintheorem}

Notice that this theorem, in particular, extends Theorem \ref{BottSam} for orbifolds in the odd dimensional case. For even dimensions, we still think that Theorem \ref{BottSam} can be generalized as follows:

\begin{cnj}\label{cnj:conjecture}
The integral (orbifold) cohomology ring of an even-dimensional, simply connected Besse orbifold is generated by one element.
\end{cnj}

For further motivation to this conjecture and for a more refined version of it we refer the reader to Section \ref{conj}.

\subsection{Idea of the proof, and structure of the article}

The first step toward the proof is to find a way to determine when an orbifold is actually a manifold, based on its orbifold cohomology. In this respect, we prove:

\begin{thm}\label{T:orb-is-man}
An orientable, compact Riemannian $n$-orbifold $\Orb$ is a manifold, if and only if the orbifold Stiefel Whitney classes $\mathsf{w}_i(\Orb)\in H^i_{orb}(\Orb;\ZZ_2)$ and the orbifold Pontryagin classes $\mathsf{p}_i(\Orb)\in H^{4i}_{orb}(\Orb;\ZZ_p)$ (for every $p$ odd prime) are nilpotent.
\end{thm}

We then focus on a Besse orbifold $\Orb$. First, an orbifold version of Wadsley's Theorem gives, as a byproduct, that the universal cover of a Besse orbifold is again a Besse orbifold. This reduces our work to the simply connected case. We then consider a loop space $\Omega^{\orb}_{p,q}\Orb$ which somehow keeps track of the orbifold structure of $\Orb$.
Such a space has already been introduced in \cite{GuHa}. Here, we give a different, though equivalent definition in terms of the frame bundle, instead of using the language of groupoids. Then we prove, via Morse theoretic arguments, that the Betti numbers of  $\Omega^{\orb}_{p,q}\Orb$ are uniformly bounded. Secondly, we show that $\Omega^{\orb}_{p,q}\Orb$ is an H-space, and we use this structure to prove that the cohomology $H^*(\Omega^{\orb}_{p,q}\Orb,F)$, with $F=\QQ$ and $\ZZ_p$, is extremely simple (i.e. isomorphic as a graded vector space to a polynomial algebra $F[x]$). Finally, this is used to prove that, in the simply connected case, the orbifold cohomology of $\Orb$ satisfies the conditions of Theorem \ref{T:orb-is-man} and thus $\Orb$ is a (Besse) manifold.

The paper is structured as follows. In Section \ref{sec:prelim} we recall basic notions and results about (Riemannian) orbifolds and Hopf algebras that are required for our presentation. In particular, in Section \ref{sec:examples} we describe the class of examples of Besse orbifolds given by weighted projective spaces in more detail. In Section \ref{sec:recognition_thm} we provide a short and self-contained proof of Theorem \ref{T:orb-is-man}. In Section \ref{sec:orbi_loop_morse} we introduce loop spaces for orbifolds and apply Morse theory to obtain the universal bounds on the Betti numbers of $\Omega^{\orb}_{p,q}\Orb$ for a Besse orbifold $\Orb$. In Section \ref{sec:rational_cohomology_ring} we apply rational homotopy theory to deduce that for an odd-dimensional simply connected Besse orbifold the rational cohomology ring of the loop space is a polynomial ring. In Section \ref{S:Zp-cohomology} we build on ideas of McCleary and Browder to determine the corresponding cohomology ring with coefficients in a finite field of prime order, at least as a vector space. In Section \ref{sec:proof_main_thm} we show via an application of the Leray--Serre spectral sequence that this information suffices to guarantee the assumptions of Theorem \ref{T:orb-is-man} and therewith prove our main result. Finally, in Section \ref{conj} we provide more background on our Conjecture \ref{cnj:conjecture}. Moreover, there is an appendix that contains complete proofs for several auxiliary statements on Hopf algebras (see Appendix \ref{sec:mani_hopf_algbras}) and for an orbifold version of Wadsley's theorem (see Appendix \ref{APP:Wadsley}).

\bigskip

\noindent\textbf{Acknowledgments.} We would like to thank Alexander Lytchak for drawing our attention to the starting question of this paper and the work \cite{Qui71} of Quillen.

The first named author was supported both by a Heisenberg grant and another research grant of the German Research Foundation; he is moreover associated to the DFG Priority Programme 2026. The second named author is partially supported by the DFG funded project SFB/TRR 191. The third named author is partially supported by the NSF grant DMS-1810913.

\section{Preliminaries} \label{sec:prelim}

\subsection{Riemannian orbifolds}\label{sub:basics_on_orbifolds}

A \emph{length space} is a metric space in which the distance between any two points can be realized as the infimum of the lengths of all rectifiable paths connecting these points \cite{MR1835418}. An \emph{$n$-dimensional Riemannian orbifold} $\Orb^n$ is a length space such that for each point $x \in \Orb$ there exists a neighborhood $U$ of $x$ in $\Orb$, an $n$-dimensional Riemannian manifold $M$ and a finite group $\Gamma$ acting by isometries on $M$ such that $U$ and $M/\Gamma$ are isometric. In this case we call $M$ a \emph{manifold chart} and $U$ a local chart of $\Orb$ around $x$. We call a chart \emph{good} if $\Gamma$ fixes a preimage of $x$ in $M$. Behind this definition lies the fact that an isometric action of a finite group on a simply connected Riemannian manifold can be recovered from the corresponding metric quotient \cite[Lem.~2.1]{La2}. By this fact every Riemannian orbifold admits a canonical smooth structure. Roughly speaking, this means that there exist equivariant, smooth transition maps between the manifolds charts (cf. \cite{MR2973378}). Conversely, every smooth orbifold admits a Riemannian metric and in this sense the two notions are equivalent.
 In the following we sometimes omit the prefix ``Riemannian'' if a certain property only depends on the smooth structure. For a point $x$ on an orbifold the linearized isotropy group of a preimage of $x$ in a manifold chart is uniquely determined up to conjugation. Its conjugacy class is denoted as $\Gamma_x$ and is called the \emph{local group} of $\Orb$ at $x$. A point $x\in \Orb$ is called \emph{regular} if its local group is trivial and otherwise \emph{singular}.

An \emph{(orbifold) geodesic} on a Riemannian orbifold is a continuous path that can locally be lifted to a geodesic in a manifold chart. A \emph{closed geodesic} is a continuous loop that is a geodesic on each subinterval. A \emph{prime geodesic} is a closed geodesic that is not a concatenation of nontrivial closed geodesics. See also \cite{La} for a more detailed discussion on orbifold geodesics. We call a Riemannian metric on an orbifold as well as a connected Riemannian orbifold \emph{Besse}, if all its geodesics ``are closed'', i.e. if they factor through closed geodesics. A Besse orbifold is geodesically complete and hence complete by the Hopf--Rinow theorem for length spaces \cite[Thm.~2.5.28]{MR1835418}.

\subsection{Examples of Besse orbifolds}\label{sec:examples} Consider the action of the unit circle $\Ss^1\subset \Co^1$ on the unit sphere $\Ss^{2n+1}\subset \Co^{n+1}$ defined by
\[
			z (z_0,\ldots,z_n)=(z^{a_0}z_0,\ldots, z^{a_n}z_n)
\]
for some weights $a_i\in \ZZ \backslash\{0\}$. Set $a=(a_0,\ldots,a_n)$. This action is almost free, i.e. its isotropy groups are finite, and the quotient space $\mathbb{CP}_a^n=\Ss^{2n+1}/\Ss^1$ with its quotient metric is a $2n$-dimensional Riemannian orbifold. It is referred to as a \emph{weighted complex projective space}. The geodesics on $\mathbb{CP}_a^n$ are the images of horizontal geodesics on $\Ss^{2n+1}$, i.e. of those geodesics that are perpendicular to the $\Ss^1$-orbits \cite{MR2529473}. In particular, all geodesics of $\mathbb{CP}_a^n$ are closed, i.e. $\mathbb{CP}_a^n$ is a Besse orbifold. However, not all of them have the same length \cite{MR2529473}. For $n=1$ the orbifold $\mathbb{CP}_a^n$ is also known as a \emph{football} or \emph{spindle orbifold}. In \cite[Sect.~2.2]{La} an infinite-dimensional family of Besse metrics on spindle orbifolds is described.

Similarly, one can show that if $\rho$ is a representation of $\SUr(2)$ on a complex vector space $\Co^n$ such that all irreducible components are complex even-dimensional, then the induced action of $\SUr(2)$ on the unit sphere in $\Co^n$ is almost free and the corresponding quotient, which is also referred to as a \emph{weighted quaternionic projective space} \cite{MR3521081}, is a Besse orbifold \cite[p.~154]{MR0950559}. Indeed, the only infinite, closed subgroups of $\SUr(2)$ are its maximal, one-dimensional tori, all of which are conjugated, and $\SUr(2)$ itself, and the restriction of an irreducible representation of $\SUr(2)$ on $\Co^n$ to its maximal torus $\Ss^1$ has weights $(n-1)-2k$, $k=0,\ldots,n-1$, so that infinite isotropy groups can only occur for odd $n$.

All the weighted projective spaces we have discussed in this section are even-dimensional and, by the long exact sequence in homotopy, simply-connected as orbifolds, i.e. they do not admit non-trivial orbifold coverings (see Sections \ref{sub:orbi_cover} and \ref{sub:orb_invariants}).

\subsection{Orbifold coverings}\label{sub:orbi_cover}

A \emph{covering} of complete $n$-dimensional Riemannian orbifolds $\Orb'$ and $\Orb$ is a submetry $p:\Orb' \To \Orb$, i.e. a map satisfying $p(B_r(x))=B_r(p(x))$ for all $x\in \Orb'$ and all $r>0$. In this case each point $x\in \Orb$ has a neighborhood $U$ isometric to some $M/\Gamma$ for which each connected component $U_i$ of $p^{-1}(U)$ is isometric to $M/\Gamma_i$ for some subgroup $\Gamma_i<\Gamma$ such that the isometries are compatible with the natural projections $M/\Gamma_i \To M/\Gamma$ \cite{La2}. In fact, for smooth orbifolds the concept was originally defined by Thurston using the latter property \cite{Thurston}. According to him an orbifold is called \emph{good} (or developable) if it is covered by a manifold. Otherwise it is called \emph{bad}. Note that a finite covering of a Besse orbifold is itself Besse. Thurston also showed that the theory of orbifold coverings works analogously to the theory of ordinary coverings \cite{Thurston}. In particular, there exist universal coverings and one can define the \emph{orbifold fundamental group} $\pi_1^{orb}(\Orb)$ of a connected orbifold $\O$ as the deck transformation group of the universal covering. The orbifold fundamental group also admits interpretations in terms of loops. To explain the one outlined in \cite{La2} we first introduce the frame bundle.

\subsection{Orbifold frame- and tangent bundles} \label{sub:orb_bundles}
Given a Riemannian orbifold $\Orb$ one can construct an \emph{(orthonormal) frame bundle} $\Fr\Orb$ by taking an open cover $\{U_i\}_i$ of $\Orb$ such that the $U_i$ admit good manifold charts $\tilde{U}_i$, and gluing the quotients $\Fr (\tilde{U}_i)/\Gamma_i$ of the orthonormal frame bundles $\Fr (\tilde{U}_i)$ by the induced actions of $\Gamma_i$ via differentials together, cf. \cite[Def.~1.22]{MR2359514}. This frame bundle is a manifold with an almost free $\Or(n)$-action ($n=\dim \Orb$) and an $\Or(n)$-equivariant projection $\pi:\Fr \Orb \to \Orb$ (where $\Or(n)$ acts trivially on $\Orb$). Moreover, $\Fr \Orb$ carries a natural metric \cite{MR0200865,MR0511689} so that the $\Or(n)$-action is isometric, and the projection $\pi: \Fr \Orb\to \Orb$ is an \emph{orbifold Riemannian submersion}, in the sense that  for every point $z$ in $\Fr \Orb$, there is a neighbourhood $V$ of $z$ such that $U=\pi(V)$ is a good chart, and $\pi|_V$ factors as $V\stackrel{\tilde{\pi}}{\longrightarrow} \tilde{U}\stackrel{\pi_U}{\longrightarrow} U$, where $\tilde{\pi}$ is a standard Riemannian submersion, inducing an isometry $\Fr \Orb /\Or(n)\to \Orb$.

 An orbifold is called \emph{orientable} if its frame bundle is disconnected. By passing to the orientable double cover \cite[Lem.~4.8]{La2} we can assume that all our Besse orbifolds are orientable. For an orientable orbifold we also sometimes simply work with the bundle of oriented orthonormal frames, on which $\SOr(n)$ acts isometrically with quotient $\Orb$, and still denote this bundle as $F \Orb$.

 Similar to the construction of the frame bundle one can construct a tangent bundle $T \Orb$ which is an orbifold and which carries a natural metric as well corresponding to the Sasaki-metric in the manifold case \cite{MR0112152,MR0145456}.

\subsection{Algebraic orbifold invariants}\label{sub:orb_invariants}

Given a regular point $p \in \Orb$ let $\Omega_{p,p} F\Orb:=\{ \gamma\in C^0(I,F \Orb) \mid \pi(\gamma(0))=\pi(\gamma(1))=p\}$ be the set of all continuous paths in $F \Orb$ that start and end at the fiber $\pi^{-1}(p)$ over $p$. One can think of an element in $\Omega_{p,p} F\Orb$ as a path in $\Orb$ together with a moving orthonormal frame. The action of $\Or(n)$ on $F \Orb$ induces an action of $\Or(n)$ on $\Omega_{p,p} F\Orb$. The fundamental group of $\Orb$ is isomorphic to the quotient of $\Omega_{p,p} F\Orb$ by the equivalence relation generated by the action of $\Or(n)$ on $F \Orb$ and by homotopies through paths in $\Omega_{p,p} F\Orb$, where the group multiplication is defined via concatenations of paths \cite[Prop.~4.15]{La2}.

A model for the so-called \emph{classifying space} $B\Orb$ of an orbifold $\Orb$ is given by the Borel-construction $F \Orb _{\Or(n)}$ of the action of $\Or(n)$ on $F \Orb$ \cite{Haef,GuHa}. More precisely, let $E\Or(n) \To B \Or(n)$ be a universal bundle of $\Or(n)$, then we can take $B \Orb$ to be the quotient of $F \Orb \times E \Or(n)$ by the diagonal action of $\Or(n)$. Moreover, if $G$ is any Lie group that acts almost freely and isometrically on a Riemannian manifold $M$ so that the quotient $\Orb=M/G$ is isometric to $\Orb$, then the corresponding Borel construction $M_G$ is weakly homotopy equivalent to $B\Orb$ (cf. \cite{La2}). The projections to the two factors induce projections $\pi_1: B \Orb \To \Orb$ and $\pi_2:B \Orb \To B \Or(n)$. Orbifold coverings of $\Orb$ are in one-to-one correspondence with coverings of $B\Orb$. In particular, the orbifold fundamental group of $\Orb$ is isomorphic to the ordinary fundamental group of $B \Orb$. Similarly, one can define other invariants of $\Orb$ in terms of $B\Orb$, e.g. the \emph{orbifold (co)homology} and \emph{homotopy groups}
\[
			H_*^{\orb}(\Orb)=H_*(B\Orb), \; H^*_{\orb}(\Orb)=H^*(B\Orb) \text{ and } \pi_*^{\orb}(\Orb,x_0)=\pi_*(B\Orb,\hat x_0),
\]
where $\hat x_0\in B \Orb$ projects to $x_0\in \Orb$. We will also work with the \emph{orbifold Stiefel Whitney classes} $\mathsf{w}_i(\Orb)\in H^i_{orb}(\Orb;\ZZ_2)$ and the \emph{orbifold Pontryagin classes} $\mathsf{p}_i(\Orb)\in H^{4i}_{orb}(\Orb;\ZZ_p)$ which are defined as the pullbacks of the respective classes of $B\Or(n)$ via the projection $\pi_2: B\Orb \To B \Or(n)$.

\subsection{Orbifold characteristic classes}

The map $F \Orb \times E \Or(n)\to B\Orb$ is, by construction, a principal $\Or(n)$-bundle. In particular, it is classified by a map $\varphi: B\Orb\to B\Or(n)$. Recall the following result:

\begin{thm}
For any odd prime $p$,
\[
H^*(B\Or(n);\ZZ_p)=\ZZ_p[\mathsf{p}_1,\ldots, \mathsf{p}_{\lfloor {n-1\over 2}\rfloor}],\qquad \textrm{where }\mathsf{p}_i\in H^{4i}(B\Or(n);\ZZ_p).
\]
Furthermore,
\[
H^*(B\Or(n);\ZZ_2)=\ZZ_2[\mathsf{w}_1,\ldots, \mathsf{w}_{n}],\qquad \textrm{where }\mathsf{w}_i\in H^{i}(B\Or(n);\ZZ_2).
\]
\end{thm}

We define the \emph{orbifold Pontryagin classes} $\mathsf{p}_i(\Orb)\in H^{4i}(B\Orb(n);\ZZ_p)=H^{4i}_{\orb}(\Orb;\ZZ_p)$ (resp. the  \emph{orbifold Stiefel-Whitney classes} $\mathsf{w}_i(\Orb)\in H^i(B\Orb(n);\ZZ_2)=H^i_{\orb}(\Orb;\ZZ_2)$) as the pullbacks of $\mathsf{p}_i\in H^{4i}(B\Or(n);\ZZ_p)$ (resp. of the $\mathsf{w}_i\in H^{i}(B\Or(n);\ZZ_2)$) under the map $\varphi: B\Orb\to B\Or(n)$.

While the orbifold cohomology with rational coefficients coincides with the usual cohomology of the underlying topological space, the orbifold cohomology with $\ZZ_p$ coefficients may have infinite cohomological dimension. For instance, for a finite subgroup $G<\Or(n)$ the orbifold cohomology of $\R^n/G$ with $\ZZ_p$ coefficients coincides with the group cohomology of $G$ with $\ZZ_p$ coefficients which is always infinite dimensional.

 Nevertheless, we have the following:

\begin{prp}\label{P:orb-coh-finitely-gen}
Let $\Orb$ be a compact, orientable orbifold. Then $H^*_{\orb}(\Orb;\ZZ_p)$ is a finitely generated module over $H:=H^*(B\Or(n);\ZZ_p)$.
\end{prp}

\begin{proof}
Since $\Orb$ is orientable, choose an orientation and take $F\Orb$ to be the bundle of \emph{oriented} orthonormal frames of $\Orb$. Nothing changes from before, except that now the map $F\Orb$ is a principal $\SOr(n)$-bundle, and the projection $F\Orb\times E\SOr(n)\to B\Orb$ is classified by a map $B\Orb\to B\SOr(n)$.  Consider the Serre spectral sequence  with $\ZZ_p$-coefficients ($p$ a prime) of the homotopy fibration $F\Orb\to B\Orb\to B\SOr(n)$. The $E_2$-page is $E_2^{p,q}=H^p(B\SOr(n);H^q(F\Orb; \ZZ_p))$. Since $\Orb$ is compact, $F\Orb$ is a compact manifold, and thus $H^q(F\Orb; \ZZ_p)$ is a vector space over $\ZZ_p$ with finite dimension (say $m$). Moreover, $E_2^{*,*}$ is isomorphic, as an $H:=H^*(B\SOr(n);\ZZ_p)$-module, to $H^p(B\SOr(n);\ZZ_p^{\oplus m})\cong H^{\oplus m}$ and, in particular, it is a finitely generated module over $H$.

We claim, by induction, that every page $(E_r,d_r)$ is a finitely generated differential module over $H$: in fact, this is true for $(E_2,d_2)$ since $H$ embeds into $E_2^{*,0}$, and by the product structure in $E_2$. Moreover, if this is true for $(E_{r-1},d_{r-1})$, then $H$ commutes with $d_{r-1}$ and induces an $H$-module structure on $E_r$. Moreover, there is a ring homomorphism $\phi_r:H\to E_r^{*,0}$ such that the action of $h\in H$ on $x\in E_r$ is given by $h\cdot x=\phi_r(h)\cup x$. In particular, since $d_r$ is a derivation and the $H$-action is via multiplication by elements in $E_r^{*,0}$, it follows that $d_r$ commutes with $H$. Moreover, as a finitely generated module over the Noetherian ring $H$, $E_{r-1}$ is a Noetherian $H$-module and thus any $H$-submodule is finitely generated. In particular, $\ker(d_{r-1})$ is finitely generated as an $H$-module, and therefore the quotient $E_r=\ker(d_{r-1})/d_{r-1}(E_{r-1})$ is finitely generated as well.

Since $F\Orb$ is finite dimensional, the spectral sequence degenerates at the $(r_0=\dim F\Orb)$-page, thus $E^{*,*}_\infty=E^{*,*}_{r_0}$ is a finitely generated $H$-module. Since $E^{*,*}_{\infty}\simeq H^{*}(B\Orb;\ZZ_p)$ as $H$-modules, the result follows.
\end{proof}

\subsection{Hopf algebras}\label{SS:Hopf}

Through the paper, we will work with the following definitions.

\begin{dfn}
An \emph{algebra} is a graded vector space $A=\bigoplus_i A_i$ over a field $\K$, together with a graded linear map $\Phi:A\otimes A\to A$, and an element $1\in A_0$ such that $\Phi(1\otimes x)=\Phi(x\otimes 1)=x$ for every $x\in A$. An algebra $A$ is  \emph{associative} if $\Phi(\Phi(x\otimes y)\otimes z)=\Phi(x \otimes \Phi(y\otimes z))$ for every $x,y,z\in A$.
\end{dfn}

As usual, we will refer to $\Phi(x\otimes y)$ simply as $x\cdot y$ or, if there is no source of confusion, simply as $xy$.

\begin{dfn}
A \emph{coalgebra} is a graded vector space $A=\bigoplus_i A_i$ over a field $\K$, together with graded linear maps $\Psi:A\to A\otimes A$ such that, for every $x\in \bar{A}=\bigoplus_{i>0}A_i$, one has $\Psi(x)=1\otimes x+1\otimes x+\sum_i x'_i\otimes x''_i$ were $1\in A_0=K$ and $x'_i,x''_i\in \bar{A}$. A coalgebra $A$ is \emph{co-associative} if $(\Psi\otimes 1)\otimes \Psi=(1\otimes\Psi)\otimes \Psi$.
\end{dfn}

\begin{dfn}
A \emph{Hopf algebra} is a graded vector space $A=\bigoplus_i A_i$ over a field $\K$, together with the structure of an algebra $\Phi:A\otimes A\to A$, $1\in A_0$; and of a coalgebra $\Psi:A\to A\otimes A$, such that the map $\Psi$ is a morphism of algebras, where $A\otimes A$ is equipped with the multiplication $(x\otimes y)\cdot (z\otimes w)=(-1)^{\deg(y)\deg(z)} (x\cdot y)\otimes (z\cdot w)$.
\end{dfn}

Given any algebra $A$, the dual space $A^*=\Hom(A,\K)$ inherits the structure of a coalgebra, by defining $\Psi: A^*\to A^*\otimes A^*\simeq(A\otimes A)^*$ by $\Psi(\bar{x})(y\otimes z):=\bar{x}(y\cdot z)$ for any $\bar{x}\in A^*$ and $y,z\in A$. In particular, if $A$ is a Hopf algebra, then $A^*$ is a Hopf algebra as well.

The requirement that $A_0=\K$ is not always made, and Hopf algebras satisfying this condition are usually called \emph{connected}. Throughout this paper, we will only deal with connected Hopf algebras so we will not make this distinction. Furthermore, we will only work with $\K=\QQ$ or $\K=\ZZ_p$, which are examples of perfect fields. In this case, the following structure theorem applies (cf. \cite[Thm~7.11]{MM}):

\begin{thm}[Borel--Hopf Theorem]\label{T:Borel-Hopf}
If $A=\bigoplus_i A_i$ is a connected Hopf algebra over the perfect field $\K$, the multiplication in $A$ is (graded) commutative, and $\dim A_i<\infty$ for all $i$ then, as an algebra, $A$ is isomorphic with a tensor product $\bigotimes_{i\in I}A_i$ of Hopf algebras $A_i$, where $A_i$ is a Hopf algebra with a single generator $x_i$.
\end{thm}

We will need the following definitions:

\begin{dfn}
Given a Hopf algebra $A$, an element $x\in \bar{A}$ is called:
\begin{itemize}
\item \emph{decomposable} if it can be written as a finite sum $x=\sum_i y_iz_i$  for some $y_i,z_i\in \bar{A}$, and \emph{indecomposable} otherwise.
\item \emph{primitive} if $\Phi(x)=1\otimes x+ x\otimes 1$.
\end{itemize}
\end{dfn}
\begin{dfn}
A Hopf algebra $A$ is called:
\begin{itemize}
\item \emph{primitive} if there exists a set of primitive generators for $A$.
\item \emph{coprimitive} if primitive elements are indecomposable.
\item \emph{biprimitive} if it is both primitive and coprimitive.
\end{itemize}
\end{dfn}

Given a Hopf algebra $A$, the algebra and coalgebra structures define filtrations $F^iA$ and $G^iA$, by:
\begin{itemize}
\item $F^0A=A$ and $F^{i+1}A=\Phi(\bar{A}\otimes F^iA)\subseteq F^iA$.
\item $G^0A=A_0=\K$ and $G^{i}A:=\ker(\bar{\Psi}_i)\supseteq G^{i-1}A$, where $\bar{\Psi}_i=A\to (\bar{A})^{\otimes (i+1)}$ is defined iteratively by
\[
\bar{\Psi}_1(x)=\Psi(\rho(x))-\rho(x)\otimes 1-1\otimes \rho(x)\quad \bar{\Psi}_{i+1}=(\bar{\Psi}_1\otimes 1\otimes\ldots\otimes 1)\circ \bar{\Psi}_i
\]
and where $\rho:A \To \bar A$ is the natural projection.
\end{itemize}

Given such filtrations, one can define the spaces $E_0(A)=\bigoplus_iF^iA/F^{i+1}A$ and $_0E(A)=\bigoplus_i G^{i}A/G^{i-1}A$, which are both isomorphic to $A$ as graded vector spaces. Moreover, both spaces inherit a Hopf algebra structure from $A$. In particular, one can define the Hopf algebra $_0EE_0(A)$, and the following holds:

\begin{thm}[\cite{Bro63}, Corollary 2.6]
If $A$ is an associative Hopf algebra, then $_0EE_0(A)$ is a biprimitive Hopf algebra.
\end{thm}

Biprimitive Hopf algebras will play a role in the context of \emph{differential Hopf algebras}, that is, Hopf algebras endowed with a differential $d:A\to A$ of degree 1, such that the maps $\Phi:A\otimes A\to A$ and $\Psi:A\to A\otimes A$ become morphisms of differential algebras (where the differential on $A\otimes A$ is given by $d(x\otimes y)=(dx)\otimes y+(-1)^{\deg(x)}x\otimes (dy)$. We have the following:

\begin{thm}[\cite{Bro63}, Theorem 3.9]\label{T:biprimitive-hopf}
Let $(A,d)$ be a biprimitive differential Hopf algebra over $\ZZ_p$. Then, %as a differential Hopf algebra,
in the category of differential Hopf algebras, $(A,d)$ is isomorphic to a tensor product
\begin{equation}
\bigotimes_{i=1}^IK_i\otimes \bigotimes_{j=1}^JL_j\otimes Q
\end{equation}
where:
\begin{itemize}
\item  $K_i=\Lambda(y)\otimes \Z_p[z]/z^p$, where $y,z$ are primitive, $|y|$ odd and $d(y)=z$.
\item  $L_j=\Lambda(y')\otimes \Z_p[z']/z'^p$, where $y',z'$ are primitive, with $|y'|$ odd and $d(z')=y'$.
\item Q is primitively generated, and $d(Q)=0$.
\end{itemize}
\end{thm}

\section{A topological condition for an orbifold to be a manifold} \label{sec:recognition_thm}

The goal of this section is to provide an algebro-topological condition to determine when an orbifold is in fact a manifold. Most of the argument below (see Proposition \ref{P:when-orb-is-man}) also follows from a more general result \cite[Theorem 7.7, p.~568]{Qui71} due to Quillen. As a service to the reader, we provide a short and self-contained reasoning relevant to our purposes.

We shall draw on the following lemma which will help to identify a special system of local coefficients used in an $E_2$-page of a Leray--Serre spectral sequence in the proof of Proposition \ref{prprec}.
\begin{lem}\label{lemnil}
Given an orientable $n$-dimensional orbifold $\Orb$ and a subgroup $\ZZ_p\subseteq \SOr(n)$, in the fibration
\begin{align*}
\SOr(n)/\Zp\hookrightarrow
(F\Orb\times E\SOr(n))/\ZZ_p\to (F\Orb\times E\SOr(n))/\SOr(n)=B\Orb
\end{align*}
the monodromy action of the fundamental group of the base, $\pi_1(B\Orb)$, on the cohomology of the fibre, $H^*(\SOr(n)/\Zp;\Zp)$, is trivial.
\end{lem}
\begin{proof}
We prove the result in slightly larger generality replacing $\SOr(n)$ by a connected, compact Lie group $G$, $\ZZ_p$ by a discrete group $H\subseteq G$ and $F\mathcal{O}$ by a connected $G$-CW-complex $X$. In this case we have the fibration
\begin{align*}
(X\times EG)/H= X_H \rightarrow  (X\times E G)/G=X_{G}
\end{align*}

We will denote elements of $X_H$ by $[x,v]_H$ for $x\in X$ and $v\in EG$. Similarly, the elements in $X_G$ will be denoted $[x,v]_G$. Clearly, we have $[x,v]_G=[g^{-1}\cdot x,g^{-1}\cdot v]_G$ for any $g\in G$, and $[x,v]_H=[h^{-1}\cdot x,h^{-1}\cdot v]_H$ for any $h\in H$.

We claim that for any element $[\gamma]\in \pi_1(X_{G})$, its action on $G/H$ is homotopic to the identity, and thus its action on $H^*(G/H)$ is trivial.

Fix $[\gamma]\in \pi_1(X_G)$, and a curve $\gamma:I\to X_G$ representing $[\gamma]$. Consider the principal bundle
\begin{align*}
G\to X\times EG\to X_G.
\end{align*}
Let $\bar{\gamma}:I\to X\times EG$ denote a lift of $\gamma$, and let $\bar\gamma(0)=(x_0,v_0)$, $\bar\gamma(1)=g_0^{-1}\cdot\bar{\gamma}(0)$.

Consider now the bundle $G/H\to X_H\to X_G$. Let $\Gamma:G/H\times I\to X_G$ be the map $\Gamma(gH,t)=\gamma(t)$, and take the lift $\hat{\Gamma}_0:G/H\to X_H$ of $\Gamma$ at $t=0$ given by $\hat{\Gamma}_0(gH)=[g^{-1}\cdot \bar{\gamma}(0)]_H$ (which identifies $G/H$ with the fiber at $\gamma(0)$). This map can be extended to
\begin{align*}
\hat{\Gamma}:G/H\times I&\to X_H\\
(gH,t)&\mapsto [g^{-1}\cdot \bar{\gamma}(t)]_H
\end{align*}
It is easy to check that this map is well defined, is a lift of $\Gamma$, and that at $t=1$ it is given by
\[
\hat{\Gamma}_1(gH)=[g^{-1}\cdot \bar{\gamma}(1)]_H=[g^{-1}g_0^{-1}\bar{\gamma}(0)]_H=\hat{\Gamma}_0(g_0gH).
\]
In particular, the action of $[\gamma]\in \pi_1(X_G)$ on $G/H$, which is defined by $\hat{\Gamma}_1$, is given by left multiplication by $g_0$. Since $G$ is connected, this action is clearly homotopic to the identity, and the proposition is proved.

\end{proof}

\begin{rem}
For a non-orientable orbifold $\Orb$ and a subgroup $\ZZ_p\subseteq \SOr(n)$ the monodromy action of $\pi_1(B\Orb)$ on the cohomology of the fiber of
\begin{align*}
\Or(n)/\Zp\hookrightarrow
(F\Orb\times E\Or(n))/\ZZ_p\to (F\Orb\times E\Or(n))/\Or(n)=B\Orb
\end{align*}
is nontrivial (here $F \Orb$ denotes the bundle of all orthonormal frames). Indeed, a loop in $B \Orb$ can be lifted to a loop in $F\Orb_{\ZZ_p}$ that starts and ends in the same fiber component, if and only if it lifts to a loop in the orientable cover $F\Orb_{\SOr(n)}$ of $B \Orb$.

\end{rem}

\begin{prp}\label{prprec}\label{P:when-orb-is-man}
An orientable, compact Riemannian $n$-orbifold $\Orb$ is a manifold, if and only if the orbifold Stiefel Whitney classes $\mathsf{w}_i(\Orb)\in H^i_{orb}(\Orb;\ZZ_2)$ and the orbifold Pontryagin classes $\mathsf{p}_i(\Orb)\in H^{4i}_{orb}(\Orb;\ZZ_p)$ (for every $p$ odd prime) are nilpotent.
\end{prp}
\begin{proof}
We begin by turning the proposition into another, equivalent one. First of all, from Proposition \ref{P:orb-coh-finitely-gen} the orbifold Stiefel--Whitney and Pontryagin classes are all nilpotent if and only if the orbifold cohomology $H^*_{\orb}(\Orb;\ZZ_p)$ is finite dimensional for all primes $p$. It is then enough to prove that an orbifold $\Orb$ is a manifold, if and only if $H^*_{\orb}(\Orb;\ZZ_p)$ is finite dimensional. Secondly, $\Orb$ is a manifold if and only if the action of $\SOr(n)$ on the frame bundle $F\Orb$ is free. Since the orbifold cohomology $H^*_{\orb}(\Orb;\ZZ_p)$ is nothing but the $\SOr(n)$-equivariant cohomology $H^*_{\SOr(n)}(F\Orb;\ZZ_p)$, we can thus rephrase the statement of the proposition in terms of the frame bundle:
\begin{center}
The $\SOr(n)$-action on $F\Orb$ is free, if and only if $H^*_{\SOr(n)}(F\Orb;\ZZ_p)$ is finite dimensional for every prime $p$.
\end{center}

Clearly if the $\SOr(n)$-action is free, then $F\Orb/\SOr(n)$ is a manifold and $H^*_{\SOr(n)}(F\Orb;\ZZ_p)=H^*(F\Orb/{\SOr(n)};\ZZ_p)$ is finite for every prime $p$. Suppose now that the $\SOr(n)$ action is \emph{not} free. Then there is some (compact) isotropy group $H$, which contains some finite cyclic subgroup $\ZZ_p$, for some prime $p$. Thus, there exists some $\ZZ_p\subseteq \SOr(n)$ with nonempty fixed point set, $(F\Orb)^{\ZZ_p}$. Fixing such a $\ZZ_p\subseteq \SOr(n)$, consider the corresponding equivariant cohomology $H^*_{\ZZ_p}(F\Orb;\ZZ_p)$, i.e. the cohomology
of the total space of the bundle
\begin{align*}
F\Orb \to F\Orb_{\ZZ_p}=(F\Orb \times E \SOr(n) )/\ZZ_p \xrightarrow{p} B \ZZ_p.
\end{align*}
The existence of a fixed point $x\in F\Orb$ of the $\ZZ_p$-action yields a section of this bundle given by $s(b)=[x,b]_G$, where we use the notation of Lemma \ref{lemnil}. Hence in this case the induced morphism on cohomology $\operatorname{id}=H^*(p\circ s)=H^*(s)\circ H^*(p)$ yields that $H^*(p)$ is injective. In particular, in this case $H^*_{\ZZ_p}(F\Orb,\ZZ_p)$ is infinite dimensional as $H^*(B \ZZ_p,\ZZ_p)$ is so.

Next, the natural map
\begin{align*}
(F\Orb\times E\SOr(n))/\ZZ_p\to (F\Orb\times E\SOr(n))/\SOr(n)=B\Orb
\end{align*}
is an $\SOr(n)/\ZZ_p$-bundle. Due to Lemma \ref{lemnil} the action of $\pi_1(B\Orb)$ on $H^*(\SOr(n)/\ZZ_p;\ZZ_p)$ is trivial, and thus the Leray--Serre spectral sequence satisfies
\begin{align*}
E_2^{p,q}=H^p_{\SOr(n)}(F\Orb,H^q(\SOr(n)/\ZZ_p;\ZZ_p))
\end{align*}
and converges to $H^{p+q}_{\ZZ_p}(F\Orb;\ZZ_p)$.

As noted above, $H^{*}_{\ZZ_p}(F\Orb;\ZZ_p)$ has infinite cohomological dimension. We claim that this implies that $H^{*}_{\SOr(n)}(F\Orb;\ZZ_p)$ has infinite dimension as well (thus proving the proposition). We prove the claim by contradiction: if $H^{*}_{\SOr(n)}(F\Orb;\ZZ_p)$ had finite cohomological dimension, then the $E_2$-page of the spectral sequence would have finite dimension, and therefore the same would be true for the $E_{\infty}$-page. This is a contradiction, since the spectral sequence converges to  $H^{*}_{\ZZ_p}(F\Orb;\ZZ_p)$ for which we showed that it is infinite dimensional.
\end{proof}

\begin{rem}
By Quillen's result \cite[Theorem 7.7, p.~568]{Qui71} and the above reasoning, the fact that $\Orb$ is a manifold if and only if $H^*_{\orb}(\Orb;\ZZ_p)$ is finite dimensional for all primes $p$ also holds for non-orientable $\Orb$.
\end{rem}

\section{Orbifold loop space, and Morse theory} \label{sec:orbi_loop_morse}

The goal of this section is to reduce the proof of Theorem \ref{T:main-theorem} to the simply connected case, and to prove the following:

\begin{thm}
For every Besse orbifold $\Orb$, the loop space $\Omega(B\Orb)$ has uniformly bounded Betti numbers, with respect to any field of coefficients.
\end{thm}

From now on, for sake of notation we will denote $\Omega(B\Orb)$ by $\Omega B\Orb$.

\subsection{The orbifold loop space}
Let us start by defining a loop space of $\Orb$, which incorporates the information about the orbifold structure. An equivalent definition in the language of groupoids is given in \cite{GuHa}.

\begin{dfn}
Let $\Orb$ be an orbifold with frame bundle $\pi:F\Orb\to \Orb$. Given two points $p,q\in \Orb$, with $p$ regular, and fixing $z\in \pi^{-1}(q)$, define the \emph{orbifold loop space} as
\[
\Omega_{p,z}^\orb\Orb=\{\gamma:I\to F\Orb\mid \gamma\textrm{ piecewise smooth, }\gamma(0)\in \pi^{-1}(p), \gamma(1)=z\}
\]
with the compact-open topology.
\end{dfn}

Notice that, changing $z$ with $z'=g\cdot z$ for some $g\in \Or(n)$, we have $\Omega_{p,z'}^\orb\Orb=g\cdot \Omega_{p,z}^\orb\Orb$, where $g$ acts on curves $\gamma:I\to F\Orb$ point-wise, by $(g\cdot \gamma)(t)=g\cdot(\gamma(t))$. In particular, there is a homeomorphism $\Omega_{p,z'}^\orb\Orb\simeq \Omega_{p,z}^\orb\Orb$, and we call either of them simply by $\Omega_{p,q}^\orb\Orb$.

An alternative definition for the orbifold loop space is
\[
\tilde{\Omega}^\orb_{p,q}\Orb=\{\gamma\in C^0(I, F\Orb)\mid \gamma(0)\in \pi^{-1}(p), \gamma(1)=z\}.
\]
The two definitions are different, however it is a standard fact that the inclusion $\Omega_{p,q}^\orb\Orb\to \tilde{\Omega}^\orb_{p,q}\Orb$ is a homotopy equivalence. This is proved, for example, with great detail in \cite{Mil}, Theorem 17.1.

Moreover, $\tilde{\Omega}^\orb_{p,q}\Orb$ has the homotopy type of a CW-complex (cf. \cite{Mil57}). Standard Morse Theory gives a more precise description of the CW-complex structure: let us equip $F\Orb$ with a Riemannian metric $g_F$, such that $\Or(n)$ acts by isometries, and the projection $\pi: F\Orb\to \Orb$ is a Riemannian submersion as discussed in Section \ref{sub:orb_bundles}. Moreover, let $E:\Omega_{p,q}^\orb\Orb\to \RR$ be the energy functional taking $\gamma:[0,1]\to F\Orb$ to
\[
E(\gamma)=\int_0^1\|\gamma'(t)\|^2_{g_F}dt.
\]

Standard variation arguments show that the critical points of $E$ are geodesics $\gamma$ from $\pi^{-1}(p)$ to $z$, with $\gamma'(0)$ perpendicular to $\pi^{-1}(p)$. Moreover, for any such critical point, $\gamma$, its index $\ind_I(\gamma)$ (number of negative eigenvalues of the Hessian of E at $\gamma$) coincides with the number of focal points to $\pi^{-1}(p)$ along $\gamma$ in $(0,1)$ counted with multiplicity, and the nullity (i.e. dimension of the kernel of the Hessian of $E$ at $\gamma$) coincides with the multiplicity of $\gamma(1)$ as a focal point of $\pi^{-1}(p)$ along $\gamma$. For $z$ generic, the nullity of $\gamma$ is zero for every critical point $\gamma$ of $E$, and standard Morse Theory (\cite{Mil}, Thm. 17.3) shows that $\Omega_{p,q}^\orb\Orb$ is homotopy equivalent to a CW-complex, with a $d$-dimensional cell attached for every critical point $\gamma$ of index $d$.
\\

The first result is to show that the information about the critical points of $E$ and their index can be read off from the orbifold geometry of $\Orb$.

Recall  that on Riemannian orbifolds one can still define geodesics and Jacobi fields: a curve $c:I\to \Orb$ is an \emph{orbifold geodesic} if, whenever $c|_{[a,b]}$ is contained in a good open set $U\subset \Orb$ covered by the Riemannian manifold $\tilde{U}\stackrel{\pi_U}{\rightarrow}U$, there exists a (manifold) geodesic $\tilde{c}:[a,b]\to \tilde{U}$ such that $\pi_U\circ \tilde{c}=c$. Similarly, an \emph{orbifold Jacobi field} along an orbifold geodesic $c$ is a section $J$ of $c^*T\Orb$ such that, whenever $c|_{[a,b]}$ is contained in a good open set $U\subset \Orb$ and is covered by a geodesic $\tilde{c}:I\to \tilde{U}$, there is a Jacobi field $\tilde{J}:[a,b]\to \tilde{c}^*T\tilde{U}$ such that $(\pi_U)_*\tilde{J}=J$. Equivalently, orbifold Jacobi fields can be described by variations of orbifold geodesics (although one needs more care with this definition).

Just as in the manifold case, given an orbifold geodesic $c:I\to \Orb$ then $c(t_0)$, $t_0\in I$, is \emph{conjugate to $c(0)$ along $c$} if there exists an orbifold Jacobi field $J$ along $c$ such that $J(0)=J(t_0)=0$. In this case, the multiplicity of $c(t_0)$ is the dimension of the space of such Jacobi fields. Finally, given an orbifold geodesic segment $c:I\to \Orb$, we define the \emph{orbifold index} of $c$, $\ind_I^{\textrm{orb}}(c)$ to be the number of orbifold conjugate points of $c(0)$ along $c$, counted with multiplicity.

\begin{prp}\label{P:lift-geods}
The projection $\pi:F\Orb\to \Orb$ induces a bijective correspondence between critical points $\gamma\in \Omega^\orb_{p,q}\Orb$ of the energy functional, and orbifold geodesics $c:I\to \Orb$ from $p$ to $q$. Moreover, if $\gamma$ and $c$ are in such a correspondence, then the index of $\gamma$, $\ind_I(\gamma)$, equals the orbifold index of $c$, $\ind_I^{\textrm{orb}}(c)$.
\end{prp}

\begin{proof}
As mentioned above, if $\gamma:I\to F\Orb$ is a critical point for $E$, then it is a geodesic in $F\Orb$ starting perpendicularly to the fiber $\pi^{-1}(p)$. Since $\pi$ is a Riemannian submersion, the geodesic $\gamma$ is a horizontal geodesic, and thus its projection $c=\pi\circ \gamma$ is an (orbifold) geodesic in $\Orb$. Moreover, just as in the manifold case, any orbifold geodesic from $p$ to $q$ can be lifted horizontally to a horizontal geodesic, ending at $z\in \pi^{-1}(q)$, and everywhere perpendicular to the fibers of $\pi$ (in particular, starting orthogonal to $\pi^{-1}(p)$). Clearly the two maps are inverses of one another, and this gives the bijection.

Let $\Lambda_{p}$ be the space of Jacobi fields along $\gamma$, given as variations through horizontal geodesics from $\pi^{-1}(p)$. The zeroes of Jacobi fields in $\Lambda_p$ count precisely the focal points of $\pi^{-1}(p)$, thus $\ind_I(\gamma)$ equals the number of Jacobi fields in $\Lambda_p$ vanishing at points in $I$, counted with multiplicity. To prove the equivalence between the indices, it is thus enough to show that $\pi: F\Orb\to \Orb$ induces, for every $t_0\in [0,1]$, a one-to-one correspondence between:
\begin{itemize}
\item Jacobi fields $J\in \Lambda_p$ with $J(t_0)=0$, and
\item Orbifold Jacobi fields $J$ along $c$ with $J(0)=J(t_0)=0$.
\end{itemize}

For each $t\in [0,1]$ there is a small time segment $I=(t-\epsilon,t+\epsilon)$ such that $\gamma(I)$ is entirely contained in an open subset $V\subset F\Orb$ for which $U:=\pi(V)$ is a good local chart, and such that $\pi:V\to U$ factors as $V\stackrel{\tilde{\pi}}{\longrightarrow} \tilde{U}\stackrel{\pi_U}{\longrightarrow} U$, where $\tilde{U}$ is the manifold chart of $U$ and $\tilde{\pi}$ is a standard Riemannian submersion. Working in such open sets shows that Jacobi fields along $\gamma$ project to orbifold Jacobi fields along $c$ and, conversely, that orbifold Jacobi fields along $c$ can be lifted to Jacobi fields along $\gamma$. Since horizontal geodesics from $\pi^{-1}(p)$ project via $\tilde{\pi}$ to geodesics in $\tilde{U}$ from $\tilde{p}=\tilde{c}(0)$, any Jacobi field $J\in\Lambda_{p}$ projects via $\tilde{\pi}_*$ to an orbifold Jacobi field along $\tilde{c}$ with $\tilde{\pi}_*J(0)=0$. Moreover, we see that orbifold Jacobi fields along $c$ vanishing at $t_0$ correspond to the Jacobi fields in $\Lambda_{p}$ vanishing at the same point.

The space $\Lambda_{p}$ is a vector space of dimension $\dim F\Orb-1$, and it splits into 2 summands:
\begin{itemize}
\item The space $\Lambda^v$, spanned by (the restriction to $\gamma$ of) the action fields $X^*$, $X\in \mathfrak{so}(n)$. This subspace has dimension $=\dim \pi^{-1}(p)$, and consists of nowhere-vanishing Jacobi fields.
\item The space $\Lambda^h$, spanned by Jacobi fields $J$  normal to $\gamma'(t)$ such that $J(0)=0$ and $J'(0)\in \nu_{\gamma(0)}(\pi^{-1}(p))$. This space has dimension $\dim F\Orb-\dim (\pi^{-1}(p))-1=\dim \Orb-1$, and $\tilde{\pi}_*$ defines an isomorphism between $\Lambda^h$ and the space of Jacobi fields along $\tilde{c}$ vanishing at $t=0$.
\end{itemize}
Suppose now that $J_1,\ldots, J_m\in \Lambda_{p}$ are linearly independent Jacobi fields such that $J_i(t_0)=0$. Then under the splitting $J_i=X_i^*+ K_i$, with $X_i^*\in \Lambda^v$ and $K_i\in \Lambda^h$, the Jacobi fields $K_i$ must be linearly independent as well: in fact, if we had $\sum a_iK_i=0$ for some nonzero coefficients $a_i$, then
\[
\sum a_i J_i=\sum a_i X_i^*+\sum a_i K_i=\sum a_i X_i^*\in \Lambda^v,
\]
but then $J(t_0)=\sum a_i J_i(t_0)=0$ would contradict the fact that Jacobi fields in $\Lambda^v$ are nowhere vanishing.

The Jacobi fields $\tilde{\pi}_*J_1,\ldots \tilde{\pi}_*J_m$ along $\tilde{c}$ satisfy $\tilde{\pi}_*J_i(t_0)=0$. Moreover, since $\tilde{\pi}_*J_i=\tilde{\pi}_* K_i$ and the $K_i$ are linearly independent, it follows that $\tilde{\pi}_*J_1,\ldots, \tilde{\pi}_*J_m$ are linearly independent as well, and this proves that $\ind_I(\tilde{c})\geq \ind_I(\gamma)$.

To prove the opposite inequality, take linearly independent Jacobi fields $J_1,\ldots, J_m$ along $\tilde{c}$ such that $J_i(0)=J_i(t_0)=0$, and let $K_1,\ldots K_m\in \Lambda^h$ denote the Jacobi fields along $\gamma$ such that $\tilde{\pi}_*(K_i)=J_i$. Then the $K_i$ are linearly independent with $K_i(0)=0$, and let $v_i:=K_i(t_0)\in T_{\gamma(t_0)}\big(\pi^{-1}(\gamma(t_0))\big)$. Letting $X_i^*\in \Lambda^v$ such that $X_i^*(t_0)=v_i$, we get linearly independent Jacobi fields $K_i-X_i^*\in \Lambda_{p}$ vanishing at $t_0$. This proves that $\ind_I(\gamma)\geq \ind_I(\tilde{c})$, as we wanted.
\end{proof}

This theorem, together with the Morse theory arguments, implies the following:

\begin{cor}\label{C:CW-Omega}
Given a Riemannian orbifold $\Orb$ and points $p,q\in\Orb$, with $p$ regular, the orbifold loop space $\Omega_{p,q}^\orb\Orb$ is homotopy equivalent to a CW-complex with one cell of dimension $d$ for every orbifold geodesic $c:I\to \Orb$ from $p$ to $q$, with $\ind_I^\orb(c)=d$.
\end{cor}

The orbifold loop space and the loop space of the classifying space are related as follows (cf. \cite[Thm.~3.2.2]{GuHa}).

\begin{prp}\label{P:weak-hom-equiv}
Let $\Orb$ be an orbifold, and $p,q\in \Orb$. If $p\in\Orb$ is chosen to be a regular point, then $\Omega_{p,q}^\orb\Orb$ is weakly homotopy equivalent to $\Omega B\Orb$.
\end{prp}

\begin{proof}
Consider the space $\mathcal{P}_p=\{\gamma\in C^0(I,F\Orb)\mid \gamma(0)\in \pi^{-1}(p)\}$. Since $p$ is a regular point, the group $\Or(n)$-action on $\mathcal{P}_p$ by $(g\cdot \gamma)(t)=g\cdot(\gamma(t))$ is free. In particular, the Borel construction $(\mathcal{P}_p)_{\Or(n)}=(\mathcal{P}_p\times E\SOr(n))/\Or(n)$ is homotopy equivalent to the standard quotient $\mathcal{P}_p/\SOr(n)$. Moreover, the map $i:\Or(n)\to \mathcal{P}_p$ sending $g$ to the constant path $\gamma(t)\equiv g\cdot z$ is an inclusion, and in fact an $\Or(n)$-equivariant homotopy equivalence.
This induces a homotopy equivalence between the quotients $\mathcal{P}_p/\Or(n)\simeq \Or(n)/\Or(n)=\{pt\}$ and thus $(\mathcal{P}_p)_{\Or(n)}$ is contractible.

The map $ev_1:\mathcal{P}_p\to F\Orb$ sending a curve $\gamma$ to $\gamma(1)$ is a fibration, with fiber $(ev_1)^{-1}(q)=\tilde{\Omega}^\orb_{p,q}\Orb$. The map $ev_1$ is also $\Or(n)$-equivariant, and thus it induces a fibration between the Borel spaces $(\mathcal{P}_p)_{\Or(n)}\to (F\Orb)_{\Or(n)}=B\Orb$, with the same fiber $\tilde{\Omega}^\orb_{p,q}\Orb$. Since the total space $(\mathcal{P}_p)_{\Or(n)}$ is contractible, it follows, for example by \cite[Prop. 4.66]{Hat}, that $\tilde{\Omega}^\orb_{p,q}\Orb\simeq \Omega B\Orb$. Since $\tilde{\Omega}^\orb_{p,q}\Orb\simeq {\Omega}^\orb_{p,q}\Orb$, we have the result.
\end{proof}

\subsection{Orbifold loop space of Besse Orbifolds}
Let us now assume that $\Orb$ is a Besse orbifold. The first result is the following:

\begin{thm}[Wadsley's Theorem]\label{prp:common_period}
Given a Besse orbifold $\Orb$, all prime geodesics have a common period. That is, there is a length $L$ such that every prime geodesic $c$ has length $\ell(c)=L/M$, for some integer $M$ depending on $c$.
\end{thm}

The proof follows closely the original proof of Wadsley for manifolds, with minor modifications. For sake of completeness, we included a proof of the orbifold version of this theorem in Appendix \ref{APP:Wadsley}. One immediate corollary is the following:

\begin{cor}\label{C:reduce-to-simply-connected}
Any Besse orbifold $\Orb$ is compact. Moreover, if the dimension is at least $2$ then its orbifold universal cover $\tilde{\Orb}$ is Besse as well. In particular, $\pi_1^\orb(\Orb)$ is finite in this case.
\end{cor}
\begin{proof} Compactness follows from the Hopf Rinow Theorem for length spaces, cf. \cite{MR1835418}, Thm.~2.5.28. Let $p$ be a regular point of $\Orb$. We represent the fundamental group of $\Orb$ as $\pi_1^{\orb}(\Orb,p)=\Omega_{p,p} F\Orb/\sim$ as discussed in Section \ref{sub:orb_invariants}. Then any element of $\pi_1^{\orb}(\Orb,p)$ is represented by a horizontal geodesic $\gamma$ in the orthonormal frame bundle $F\Orb$ starting and ending over $p$, and hence by a geodesic loop $c$ in $\Orb$ based at $p$. Conversely, such a loop lifts uniquely up to the $\Or(n)$-action to a horizontal geodesic with end points over $p$, cf.~Proposition \ref{P:lift-geods}.

Let $L$ be the minimal common period of all closed geodesics on $\Orb$. We claim that a geodesic $c$ based at $p$ of length $L$ represents the same element in $\pi_1^{\orb}(\Orb,x)$ as its inverse. Let $\gamma$ be a lift of $c$ to $\Fr \Orb$ and let $g \in \Or(n)$ such that $g \gamma(1)=\gamma(0)$. Continuously deforming the initial vector from $v$ to $-v$ through horizontal vectors defines a homotopy through horizontal geodesics between $\gamma$ and $(g\gamma)^{-1}=g(\gamma)^{-1}$. In particular, we have $[\gamma]=[\gamma]^{-1}$ in $\pi_1^{\orb}(\Orb,x)$ as claimed.

Consequently, any element in $\pi_1^{\orb}(\Orb,p)$ is represented by a horizontal geodesic of length $<2L$. Let $\tilde \Orb \To \Orb$ be the universal covering and let $F\tilde \Orb \To F\Orb$ be the induced covering \cite[Lem.~4.6]{La2}. It follows that the fibers of $F\tilde \Orb \To F\Orb$ have diameter $\leq 2(2L+\mathrm{diam}(\pi^{-1}(p))$. Since $F\tilde \Orb$ is complete and connected and the fibers are discrete, they are finite and hence so is $\pi_1^{\orb}(\Orb,p)$.
\end{proof}

The corollary above reduces the proof of Theorem \ref{T:main-theorem} to the simply connected case. Furthermore, we can use Wadsley's Theorem and Corollary \ref{C:CW-Omega} to prove the following:

\begin{thm}\label{T:uniformly-bounded}
Given a Riemannian orbifold $\Orb$ with all geodesics closed, the loop space $\Omega B\Orb$ has uniformly bounded Betti numbers.
\end{thm}
\begin{proof}
The proof is virtually identical to the one in the manifold case, cf. Sections 7.41 and 7.42, p.~193 in \cite{Besse}. For the sake of completeness, and because this result is central in proving our Main Theorem, we recall the arguments here.

In the following we fix a regular point $p$ in $\Orb$ and choose some point $q$ that is not conjugate to $p$ along any geodesic (this can be done by a straightforward consequence of Sard's theorem applied to the normal exponential map from the normal bundle of $\pi^{-1}(p)$ to $F\Orb$). By Proposition \ref{P:weak-hom-equiv} it is sufficient to show that the Betti numbers of $\Omega_{p,q}^\orb\Orb$ are uniformly bounded. Moreover, by Corollary \ref{C:CW-Omega}, it is enough to prove that for every $d$, the number of orbifold geodesics $c:I\to \Orb$ from $p$ to $q$ with $\ind_I^\orb(c)=d$, is bounded by a constant independent of $d$.

By Wadsley's Theorem, all closed geodesics have length $L/M$ for some fixed $L$ and $M$ an integer, and there is an open dense set of geodesics with length $L$.

Notice that any closed orbifold geodesic from $p$ with length $L$ must have the same index, and nullity $n-1$, the maximum possible. In fact, on the one hand, every orbifold Jacobi field $J$ along $c$ with $J(0)=0$ and $J'(0)\perp c'(0)$ defines a variation through geodesics from $p$, which by Wadsley's Theorem must return to $p$ at time $L$, and thus $J(L)=0$, thus proving the second statement. On the other hand, letting $c_0$, $c_1$ be any two geodesics from $p$ of length $L$, and letting $\{c_s\}_{s\in [0,1]}$ be a smoothly-varying family of orbifold geodesics $c_s:I\to \Orb$ of length $L$ from $c_0$ to $c_1$, it is well-known that the only discontinuities of the function $s\mapsto \ind_I^{\orb}(c_s)$ occur at times $s_i$ in which the nullity of $c_{s_i}$ jumps up. However, since the nullity of any such $c_s$ is $n-1$, it follows that $\ind_I^\orb(c_s)$ is in fact constant. In particular, any two orbifold geodesics of length $L$ have the same orbifold index, call it $C$.

Since $\Orb$ is compact, and $p$ and $q$ are not conjugate to each other, for any $e\in \RR$ there are only finitely many geodesics between $p$ and $q$ with energy $\leq e$. Let $\{c_1, \ldots c_k\}$ be the geodesics from $p$ to $q$ with length $\leq L$.

By Wadsley's Theorem, the orbifold geodesics with length in $(mL, (m+1)L)$ must be $\{c_1^{(m)},\ldots, c_k^{(m)} \}$, where $c_i^{(m)}$ is obtained by following the geodesic $c_i$ past $q$, and up to time $mL+\ell(c_i)$ (we have $c_i^{(m)}(mL+\ell(c_i))=c_i^{(m)}(\ell(c_i))=q$). In particular, for any integer $m$ there are exactly $k$ geodesics with length in $(mL, (m+1)L)$.

Since any geodesic longer than $L$ picks up $n-1$ conjugate points at times multiple of $L$, we obtain that
\[
\ind_I^\orb(c_1^{(m)})=m(C+ n-1)+\ind_I^\orb c_i
\]
Since the $c_i$ have length $\leq L$ and the index grows monotonously with the length, we have $\ind_I^\orb(c_i)\leq C<C+n-1$ and $\ind_I^\orb(c_1^{(m)})\in \big[m(C+n-1), (m+1)(C+n-1)\big]$.
It follows that for any integer $m$, the orbifold geodesics with index in $\big[m(C+n-1), (m+1)(C+n-1)\big]$ are exactly $c_1^{(m)},\ldots ,c_k^{(m)}$. In particular, for any $d$ the number of orbifold geodesics with index $d$ is at most $k$, hence the result.
\end{proof}

\section{Computing $H^*(\Omega B\Orb;\QQ)$, for simply connected Besse orbifolds} \label{sec:rational_cohomology_ring}

In this section and the next, $\Orb$ will denote an odd-dimensional, simply connected Besse orbifold. In particular, $\Orb$ is orientable and we only work with the bundle of oriented orthonormal frames. From the previous section, the loop space $\Omega B\Orb$ has uniformly bounded Betti numbers. Using the extra information that $\Orb$ is odd dimensional, and the fact that $\Omega B\Orb$ is an H-space, we will show that, in fact, $H^*(\Omega B\Orb;\K)\simeq \K[x]$ is a polynomial algebra in one generator of even degree, whenever $\K=\QQ$ or $\ZZ_p$. We will deal with two cases separately: in this section we will treat the case $\K=\QQ$; and in the next section, we will look at the case $\K=\ZZ_p$.

The main structure that we will exploit is the fact that $\Omega B\Orb$ (here thought as the space of continuous pointed loops at some fixed point $z\in B\Orb$) is an $H$-space, with respect to the map $\Omega B\times \Omega B\to \Omega B$ given by concatenation of paths. In particular, the cohomology $H^*(B\Orb;\K)$ (with respect to any field of coefficients $\K$) is a commutative, associative Hopf algebra.

When dealing with rational coefficients, we use Sullivan's theory of minimal models. The methods and results of this section are fairly standard.

Recall that to any simply connected topological space $X$, one can associate a differential graded algebra $(\Lambda V_X,d)$: that is, a free graded commutative algebra over $\QQ$
\[
\Lambda V_X=\wedge V_X^{odd}\otimes \QQ[V_X^{even}]
\]
where $V_X=V_X^{odd}\oplus V_X^{even}$ is a vector space over $\QQ$, together with a differential $d$ on $\Lambda V_X$, of degree 1. This differential graded algebra, called \emph{minimal model of $X$}, is defined by the properties that:
\begin{enumerate}
\item Letting $\Lambda^+V_X$ be the subspace of $\Lambda V_X$ of positive degree, then $d(\Lambda V_X)\subseteq \Lambda^+ V_X\cdot \Lambda^+ V_X$.
\item $H^*(\Lambda V_X,d)\cong H^*(X;\QQ)$.
\end{enumerate}
Among the properties of the minimal model, one has that for simply connected $X$, the degree-$d$ part $V_X^d$ of $V_X$ satisfies $V_X^d\simeq \textrm{Hom}(\pi_d(X),\QQ)$.

Although we are mainly concerned with the odd-dimensional case, let us formulate the following result also for even dimensions.
\begin{prp}\label{ratellprop}
Suppose $\Orb$ is a simply connected orbifold such that $\dim H^q(\Omega B\Orb;\QQ)$ is uniformly bounded. Then:
\begin{itemize}
\item If $\dim \Orb$ is odd, then $\Orb$ has the rational type of an odd dimensional sphere, and $H^*(\Omega \Orb;\QQ)\cong \QQ[x']$ where $\deg x'=n-1$.
\item If $\dim \Orb$ is even, then $H^*(\Orb;\QQ)$ is singly generated and $H^*(\Omega \Orb;\QQ)\cong \wedge (y') \otimes  \QQ[x']$ with $\dim \Orb=\deg x'-\deg y'+1$.
\end{itemize}
\end{prp}
\begin{proof}
First of all, recall that the natural map
\[
B\Orb=(F\Orb\times E\SOr(n))/\SOr(n)\to F\Orb/\SOr(n)= \Orb
\]
induced by projection on the first factor is a rational homotopy equivalence (it induces isomorphisms in rational cohomology, and thus in rational homotopy as well). Therefore $\Orb$ and $B\Orb$ are rationally homotopy equivalent, and thus so are $\Omega \Orb$ and $\Omega B\Orb$. In particular, $H^*_\orb(\Orb;\QQ)=H^*(B\Orb;\QQ)\simeq H^*(\Orb;\QQ)$. For the rest of this section, we will use $\Orb$ and $\Omega\Orb$ instead of $B\Orb$ and $\Omega B\Orb$.

Because $H^*(\Omega \Orb;\QQ)$ is a Hopf algebra, it follows from the Borel--Hopf Theorem (Theorem \ref{T:Borel-Hopf}) that $H^*(\Omega \Orb;\QQ)$ is in fact a free graded algebra, and in particular the minimal model for $\Omega \Orb$ is isomorphic to $\big(H^*(\Omega \Orb;\QQ), 0\big)$.

Let the minimal model for $\Orb$ be of the form $(\Lambda V=\wedge(x_1,\ldots, x_r)\otimes \QQ[y_1,\ldots, y_s],d)$, where the $x_i$ have odd degree and the $y_i$ have even degree. In this case,
\[
V=\textrm{span}(x_1,\ldots, x_r, y_1,\ldots y_s).
\]
Since $V^d\simeq \pi_d(\Orb)\otimes \QQ$ and the homotopy groups of $\Omega \Orb$ are just the ones of $\Orb$ scaled by one, it follows that the minimal model of $\Omega \Orb$ can be also written as
\[
(\Lambda \overline{V},0)=(\QQ[x'_1,\ldots, {x}'_r]\otimes \wedge({y}'_1,\ldots , {y}'_s),0)\qquad \deg {x}'_i=\deg x_i-1, \,\deg {y}'_i=\deg y_i-1.
\]

By the previous section, the cohomology groups of $\Omega \Orb$, which coincide with the graded summands of $\Lambda \overline{V}$, have uniformly bounded dimension. This can only happen if the number of even degree generators of $\Lambda \overline{V}$ is at most 1 (otherwise $\Lambda \overline{V}$ would contain a polynomial algebra of the form $\QQ[{x}'_1,{x}'_2]$ for which the dimension of each degree grows unbounded). In other words, $r\leq 1$.

On the other hand, since $r$ denotes the number of \emph{odd degree} generators for the minimal model $\Lambda V$ of $\Orb$, it follows that $r\geq 1$: in fact, if there were no odd degree generators, we would have $\Lambda V=\QQ[y_1,\ldots, y_s]$, and $d=0$ for degree reasons, which would give $H^*(\Orb;\QQ)=H^*(\Lambda V,d)=(\QQ[y_1,\ldots, y_s],0)=\Lambda V=\QQ[y_1,\ldots, y_s]$, contradicting the finite dimensionality of $H^*(\Orb;\QQ)$.

Therefore, $\Lambda V=\wedge(x)\otimes \QQ[y_1,\ldots, y_s]$. We now claim that $s\leq 1$. In fact, define the \emph{pure Sullivan algebra associated to $(\Lambda V,d)$} by  $(\Lambda V, d_\sigma)$, where $d_\sigma(y_i)=0$ and $d_\sigma(x)\in \QQ[y_1,\ldots, y_s]$ defined by
\[
\operatorname{im}(d-d_\sigma)\subseteq \QQ[y_1,\ldots, y_s]\otimes \wedge^{>0}(x)
\]
By \cite{FHT}, Proposition 32.4, $H^*(\Lambda V,d)$ is finite dimensional if and only if $H^*(\Lambda V,d_\sigma)$ is finite dimensional. However, we have
\[
H^*(\Lambda V,d_\sigma)=\QQ[y_1,\ldots , y_s]/(P),\qquad\textrm{ where }d_{\sigma}x= P(y_1,\ldots, y_s)
\]
This can only be finite dimensional if $s\leq 1$, and we have two cases for $(\Lambda V,d)$:
\begin{enumerate}
\item Either $s=0$ and $(\Lambda V,d)=(\wedge(x),0)$.
\item Or $s=1$, and the only possibility is $(\Lambda V,d)=(\wedge(x)\otimes\QQ[y],dx=y^m, dy=0)$ for $m=\dim \Orb/ \deg y+1$.
\end{enumerate}
These two cases now correspond to the two cases in the assertion, i.e.~to the different parities of $\dim \Orb$. Indeed, the model in (i) is of odd dimension, the one from (ii) is even dimensional. Passing to the corresponding loop space models actually finishes the proof. It only remains to remark that the loop space cohomology in both cases obviously has universally bounded Betti numbers.
\end{proof}

\section{Computing $H^*(\Omega B\Orb;\ZZ_p)$ ($p$ prime), for simply-connected Besse orbifolds}\label{S:Zp-cohomology}

In this section we consider the cohomology of $\Omega B\Orb$ with coefficients in $\ZZ_p$, $p$ a prime. The methods of this sections follow closely the arguments of McCleary in \cite{McCl}. The whole section is devoted to proving the following:

\begin{prp}\label{P:HBOmOZp}
Let $\Orb$ be a simply connected Besse orbifold of odd dimension $n$. Then $H^*(\Omega B\Orb;\ZZ_p)\simeq \ZZ_p[x_{n-1}]$ as graded vector spaces over $\ZZ_p$, for some generator $x_{n-1}$ of degree $n-1$.
\end{prp}

As noted by Browder, there is a spectral sequence of Hopf algebras $(B_r, d_r)$ (the \emph{Bockstein spectral sequence}) with $B_1=H^*(\Omega B\Orb;\Zp)$ and $B_{\infty}\simeq (H^*(\Omega B\Orb;\ZZ)/\textrm{torsion})\otimes \ZZ_p$, and where all differentials have degree $1$. By the results in the previous section, $(H^*(\Omega B\Orb;\ZZ)/\textrm{torsion})\simeq \ZZ[x_{n-1}]$ and thus $$B_{\infty}\simeq \ZZ_p[x_{n-1}],$$ so we need to prove that $d_r=0$ for all $r$. We prove this by showing that $(B_r)^{odd}=0$ for all $r$, so that the result follows by ``lacunary principle''.

Recall that by Browder (cf. \cite{Bro61}), for every $B_k$ there is a spectral sequence of spectral sequences $_sEE_r(B_k)$, with the following properties:
\begin{enumerate}
\item $_0EE_0(B_k)$ is a ``biprimitive form'' of $B_k$---in particular, $_0EE_0(B_k)$ is a biprimitive Hopf algebra (cf. definitions in Section \ref{SS:Hopf}) isomorphic to $B_k$ as a graded vector space.
\item For every $r, k$ fixed, there are differentials $d_s$ such that $(_sEE_r(B_k),d_s)$ is a spectral sequence, converging to $_0EE_{r+1}(B_k)$.
\item For every $k$ fixed, there are differentials $d_r$ such that $(_0EE_r(B_k),d_r)$ is a spectral sequence converging to $_0EE_0(B_{k+1})$.
\end{enumerate}

Since each $_sEE_r(B_k)$ is biprimitive by \cite[Lem.~3.2]{Bro63}, by Theorem \ref{T:biprimitive-hopf} it is isomorphic, as a differential Hopf algebra, to a product
\begin{equation}\label{E:biprimitive-ss}
\bigotimes_{i=1}^IK_i\otimes \bigotimes_{j=1}^JL_j\otimes Q
\end{equation}
where:
\begin{itemize}
\item  $K_i=\wedge(y)\otimes \Z_p[z]/(z^p)$, where $y,z$ are primitive, $|y|$ odd and $d_sy=z$.
\item  $L_j=\wedge(y')\otimes \Z_p[z']/(z'^p)$, where $y',z'$ are primitive, with $|y'|$ odd and $d_sz'=y'$.
\item Q is primitively generated, and $d_sQ=0$.
\end{itemize}

Recall that the set of primitive elements in a Hopf algebra $A$ is a vector space. We will see that, if $A$ is biprimitive, then any homogeneous basis of this subspace forms a set of generators of $A$ in the tensor product description above.
\\

We now prove that $(B_k)^{odd}=0$ for all $k$, by contradiction. Supposing that some $B_k$ contains elements of odd degree, we want to produce a subalgebra of $B_1=H^*(\Omega B\Orb,\ZZ_p)$ whose dimension grows unbounded with respect to the degree, contradicting Theorem \ref{T:uniformly-bounded}. We do this by applying the following:

\begin{lem}\label{L:growth}
Suppose $A$ is a (associative and commutative) biprimitive Hopf algebra over $\ZZ_p$, and suppose that we have two linearly independent sequences $a=\{a_1,a_2,\ldots\}$, $b=\{b_1,b_2,\ldots\}$ of homogeneous primitive elements of $A$, such that the subalgebras $A_a$, $A_b$ generated by $a$ and $b$ respectively, are nonzero in degrees forming arithmetic sequences. Then the dimension of $A$ grows unbounded in the degree.
\end{lem}
\begin{proof}
Since $A$ is biprimitive, it has the form $A=\bigotimes_{i}\Zp[x_i]/(x_i^{d_i}),$ where the $x_i$'s are primitive elements and $d_i$ is either 2 if $\deg(x_i)$ is odd, or $p$ if $\deg(x_i)$ is even, cf. for example the discussion in \cite[Corollary 2.2]{Bro63}. Let $P\subset \bar{A}$ be the set of primitive elements, and let $P\to F^1A/F^2A\simeq \textrm{span}\{x_i\}$ be the projection map. This is both surjective (since $A$ is primitive) and injective ($A$ is coprimitive) and therefore $P= \textrm{span}\{x_i\}$. We claim that for any basis of $P$ consisting of homogeneous primitive elements $\{u_1, u_2, \ldots\}$, we can write
\begin{equation}\label{E:Model-of-A}
A=\bigotimes_{i}\Zp[u_i]/(u_i^{d_i})
\end{equation}
It is enough to prove it for a given degree: given $\{u_1,\ldots, u_m\}$ and $\{x_1,\ldots, x_m\}$ bases of $P_d$ (degree $d$ summand of $P$), we need to show that the algebra generated by $$\{u_1,\ldots u_m, x_{m+1}, x_{m+2}, \ldots\}$$ is of the same shape as above. Also, this is clear in odd degrees, due to the fact that the product is skew commutative. In even degrees, this is due to the fact that, with $\ZZ_p$ coefficients, the map $A\to A$, $x\mapsto x^p$ is an algebra homomorphism: the kernel of the natural map $\phi: \ZZ_p[x_1,\ldots, x_m]\to A$ is generated by $x_i^p$ and thus, letting $u_i=\sum_{ij}\alpha_{ij}x_j$ for some $A=(\alpha_{ij})\in \mathrm{GL}(m,\ZZ_p)$ and letting $B=(\beta_{ij})$ the inverse of $A$, we have that the kernel of $\phi$ is the ideal generated by $\sum_j(\beta_{ij}u_j)^p=\sum_j\beta_{ij}u_j^p$, which is the same as the ideal generated by the $u_j^p$.

Therefore, given the linearly independent set $S=\{a_1, b_1, a_2, b_2,\ldots\}$ of homogeneous elements of $P$, this can be extended to a basis $\{u_1, u_2, \ldots\}$ of homogeneous elements of $P$, such that $A$ can be written as in \eqref{E:Model-of-A}. In particular, the algebras $A_a$, $A_b$ are of this form as well, and the algebra generated by $S$ is isomorphic to $A_a\otimes A_b$. Since $A_a$, $A_b$ have elements in degrees forming arithmetic sequences, it is easy to check that the dimensions of the degree components of $A_a\otimes A_b$ (and thus the ones of $A$) grow unbounded.
\end{proof}

One sequence, satisfying the properties of the lemma above is readily available. Notice in fact that, since $B_{\infty}$ is a polynomial ring, its biprimitive form is

\[
_0EE_0(B_\infty)=\bigotimes_{i=1}^\infty\Zp[z_i]/(z_i^p), \qquad |z_i|=(n-1)p^i
\]

The elements $z_i$ can be lifted to a set of primitive elements
\begin{equation}\label{E:a}
a=\{\tilde z_1, \tilde z_2,\ldots\}\subseteq{} _0EE_0(B_1),
\end{equation}
generating a subalgebra $A_a=\bigotimes_{i=1}^\infty\Zp[\tilde{z}_i]/(\tilde{z}_i^p)$ which is nonzero in degrees multiples of $(n-1)p$ and thus satisfies the condition of the lemma.
\\

To finish the proof, we thus need to produce the second set $b$. We will consider two cases.

\subsection{Case I: the spectral sequence $_sEE_r(B_k)$ does not contain factors of type $L_j$}

Recall that we are assuming, by contradiction argument, that there exists some element of odd degree in some $B_r$, and then also in $_0EE_0(B_r)$. In this case the following lemma applies for $(s_0,r_0,k_0)=(0,0,r)$.
\begin{lem}\label{L:noLj}
Suppose that the biprimitive spectral sequence $_sEE_r(B_k)$ only contains factors of type $K_i$ and $Q$. If there is an element $y\in{}_{s_0}EE_{r_0}(B_{k_0})$ of odd degree, then there exists a sequence $b=\{x_1,x_2,\ldots\}$ of primitive elements of degree $|x_i|=(|y|+1)p^i$ in $_0EE_0(B_1)$, generating a subalgebra $A_b$ isomorphic, as a vector space, to a polynomial algebra $\Zp[w]$, $|w|=|y|+1$. Moreover, every $x_i$ gets killed at some stage in the spectral sequence $_sEE_r(B_k)$.
\end{lem}
\begin{proof}
We can assume that $y$ is primitive. Since $B_{\infty}$ does not have elements of odd degree, the element $y$ must be killed at some stage $(s_1,r_1,k_1)$. Since there are no factors of type $L_j$, the only way this can happen is that there is some $x_1\in{}_{s_1}EE_{r_1}(B_{r_1})$ such that $dy=x_1$. Since the differential commutes with the coproduct, the element $x_1$ is primitive as a differential of a primitive element. Moreover, $|x_1|=|y|+1$, and in ${}_{s_1+1}EE_{r_1}(B_{r_1})$ there is a new element of odd degree being created, namely $y_2=[yx_1^{p-1}]$. By Lemma \ref{L:primitive} below, $y_2$ is primitive, and $|y_2|=|y|+(p-1)(|y|+1)=p(|y|+1)-1$. By the same reason as before, at some later time $(s_2,r_2,k_2)$ we can find $x_2\in{}_{s_2}EE_{r_2}(B_{r_2})$ such that $dy_2=x_2$. Again, $x_2$ is primitive and $|x_2|=|y_2|+1=p(|y|+1)$. The lemma now follows by inductively repeating the same steps above and by lifting the $x_i$ to $_0EE_0(B_1)$ via an application of Lemma \ref{L:lift}.
\end{proof}

Because the elements in the sequence $b$ in Lemma \ref{L:noLj} above do not survive to $_0EE_0(B_{\infty})$, they are linearly independent of the elements in the sequence $a$ in \eqref{E:a}. Lemma \ref{L:growth} can then be applied, contradicting the uniform boundedness of $_0EE_0(B_1)$. This proves Proposition \ref{P:HBOmOZp} in this first case.

\subsection{Case II: the spectral sequence $_sEE_r(B_k)$ contains a factor of type $L_j$}

We can now assume that at some stage in the biprimitive spectral sequence $_sEE_r(B_k)$, we have a factor of type $L_j$. In other words, there is an element $x_{s,r}\in{}_sEE_r(B_k)$ of even degree (call it $D$), such that $dx_{s,r}\neq 0$. By the explicit form \eqref{E:biprimitive-ss} of each stage $_sEE_r(B_k)$, we can choose $x_{s,r}$ to be primitive. According to Lemma \ref{L:lift} the element $x_{s,r}$ is in fact represented by a primitive element $x_{0,0}\in {_0EE_0}(B_k)$.
Using the explicit shape of the biprimitive form ${_0EE_0}(B_k)$ shows that the primitive element $x_{0,0}$ can be expressed as a linear combination of indecomposable generators $\xi_i$ of ${_0EE_0}(B_k)$ as in \cite{Bro63}, Page 158. We can assume that $x_{0,0}$ is represented by one such generator $\xi_0$ (still guaranteeing our assumption $dx_{s,r}\neq 0$). By \cite[Theorem 2.7]{Bro63} $\xi_0$ is represented by an iterated $p$-th power $x_0^{p^m}$ of an indecomposable generator $x_0$ of $E_0(B_k)$. Since $B_k\simeq E_0(B_k)$ as algebras, $\xi_0=x_0^{p^m}$ is represented by an iterated $p$-th power $\xi=x^{p^m}$ of some indecomposable element $x\in B_k$. However, if $m\neq 0$, we would have $d\xi=0$ in $(B_k,d)$ and, by standard facts about the spectral sequence of a filtered complex, it would follow that $dx_{s,r}=0$ for all $x_{s,r} \in{}_sEE_r(B_k)$ represented by $\xi$. This gives a contradiction, and hence $x_{s,r}$ is represented by the indecomposable element $x$.
We claim that $dx$ cannot be written as $dw$ for some decomposable element $w$: in fact, if this was the case, then $d(x-w)=0$ would imply that every element in $_sEE_r(B_k)$ represented by $x-w$ has zero differential. On the other hand, since $w\in F_2(B_k)$ (second subset in the filtration of $B_k$, which coincides with the set of decomposable elements), then $x_0=\{x\}=\{x-w\}$ in $E^1_0(B_k)$, and this implies that $x_{s,r}$ would also be represented by $x-w \in F_1(B_k)$, a contradiction since $d x_{s,r}\neq 0$.

Recall (cf. \cite{McCl}) that the homology $H_*(\Omega B\Orb;\ZZ_p)$ is a Hopf algebra as well, and that there is a homology Bockstein spectral sequence $(B^k,d^k)$ with $B^1=H_*(\Omega B\Orb;\ZZ_p)$ and such that each $B^k$ is a Hopf Algebra, which is identified with the dual of $B_k$. In particular, since $dx\in B_k$ is linearly independent from every $dw$, $w$ decomposable, it then follows that there exists an element $\bar{y}\in B^k$ such that $\bar{y}(dx)\neq 0$ and $\bar y(dw)=0$ for every $w$ decomposable. Then:
\begin{itemize}
\item $d\bar{y}(w)=0$ for all decomposable $w$. Hence, $d \bar{y}$ is primitive by Lemma \ref{L:decomposable-primitive}.
\item $d\bar{y}(x)\neq 0$, which implies that $\bar{x}:=d\bar{y}\neq 0$.
\end{itemize}
From this, it follows from Theorem 3 of \cite{McCl} that $\bar{x}$ produces \emph{$\infty$-implications}. That is, there exists a sequence of primitive elements $\bar{x}_1, \bar{x}_2,\ldots$ in homology, with $\bar x_i\in B^{r_i}$ for some $r_i$, and $|\bar{x}_i|=D p^i$.
Moreover, it follows from the proof of Theorem 3 of \cite{McCl} that:
\begin{enumerate}
\item For every $i$, there exists an element $\bar y_i\in B^{r_i}$ such that $d\bar y_i=\bar x_i$.
\item Either $r_{i+1}=r_i$ and $\bar x_{i+1}=\bar x_i^p$ (hence $\bar{x}_{i+1}=d(\bar{y}_i\bar x_i^{p-1})$), or $r_{i+1}=r_i+1$ and $\bar y_{i+1}$ is represented by $\bar y_i\bar x^{p-1}$ (hence $\bar{x}_{i+1}=d[\bar{y}_i\bar x_i^{p-1}]$).
\end{enumerate}
Choose elements $x_i\in B_{r_i}$ such that $x_i(\bar{x}_i)\neq 0$, which implies $dx_i\neq 0$. By Lemma \ref{L:decomposable-primitive} we can choose the $x_i$'s to be indecomposable in $B_{r_i}$. These $x_i$ represent indecomposable elements in $E_0 B_{r_i}$ that lie in $E^1_0 B_{r_i}$. In particular, they represent primitive elements in $E_0 B_{r_i}$ by \cite[Proposition~1.3, (i)]{Bro63}, and hence primitive elements in $_0EE_0(B_{r_i})$ by \cite[Proposition~1.2, (iv)]{Bro63}, which, by Lemma \ref{L:lift}, can be lifted to primitive elements in  $_0EE_0(B_r)$, again called $x_i$. Moreover, since $|x_i|=D p^i$, it follows by dimensional reasons that the algebra
\[
A_b=\bigotimes_i \Zp[x_i]/(x_i^p)\subseteq {}_0EE_0(B_1)
\]
that is, the algebra generated by the $x_i$, is isomorphic as a vector space to a polynomial algebra in one generator of degree $D$.

Since the elements in the sequence $b$, i.e.~the elements $x_i$, do not survive to $_0EE_0(B_{\infty})$, they are linearly independent of the elements in the sequence $a$ in \eqref{E:a}. Once again Lemma \ref{L:growth} can then be applied, contradicting the uniform boundedness of $_0EE_0(B_1)$.

This proves Proposition \ref{P:HBOmOZp} in the second and final case.

\section{Proof of the Main Theorem}\label{sec:proof_main_thm}

Let $\Orb$ be an odd-dimensional Besse orbifold. By Corollary \ref{C:reduce-to-simply-connected}, the universal cover $\tilde{\Orb}$ is a Besse orbifold as well. We first prove that $\tilde{\Orb}$ is in fact a manifold.

Consider the Leray-Serre spectral sequence of the path-fibration $\Omega B\tilde{\Orb}\to PB\tilde{\Orb}\to B\tilde{\Orb}$ with $\ZZ_p$ coefficients. The second page is $E^{r,s}_2= H^r(B \tilde{\Orb};\Z_p) \otimes H^s(\Omega B \tilde{\Orb} ;\Z_p)$. From Proposition \ref{P:HBOmOZp}, we know that $H^*(\Omega B \tilde{\Orb} ; \Z_p) \cong \Z_p[e_{n-1}]$ as a graded vector space. Since the total space of the fibration is contractible, the spectral sequence converges to the cohomology of a point. In particular, the lowest non-zero degree in which $B\Orb$ has nontrivial $\Z_p$ cohomology is $n$ and the dimension of this cohomology group is $1$. Hence, for $r\leq n$ we have $E^{r,s}_2= \Lambda(e_{n}) \otimes \Z_p[e_{n-1}]$ as graded vector spaces, and $e_n\cdot e_{n-1}\neq 0$ in $E_2^{p,q}$. Since $x=e_{n}$ has even degree we find using the Leibniz rule that
\[
		d_{n}(x^k)= k x^{k-1}  d_{n}(x).
\]
It follows inductively that for $k<p$ we have $x^k\neq 0$ and $d_{n}(x^k) \neq 0$. This implies that $H^i(B\tilde{\Orb};\Z_p)=0$ for all $0<i<p(n-1)+1$, $i\neq n$.

First suppose that $p>2$. Because $H^{4i}(B\tilde{\Orb};\ZZ_p)=0$ for $i= 1,\ldots [n/2]$, it follows that the orbifold Pontryagin classes $\mathsf{p}_i(\tilde{\Orb})$ are all zero.

In the case $p=2$, since $H^{<n}(B\tilde{\Orb};\ZZ_p)=0$, all the orbifold Stiefel Whitney classes are trivial with the possible exception of $w_{n}(\tilde{\Orb})$. However, it follows from Wu's formula that the Stiefel Whitney classes are generated over the Steenrod algebra by those in degrees that are a power of two. Since $n>2$ is not a power of two, $w_{n}(\tilde{\Orb})$ is trivial as well.

From Proposition \ref{P:when-orb-is-man}, it follows that $\tilde{\Orb}$ is a Besse manifold. Since $\tilde{\Orb}$ is odd dimensional, it follows from Theorem \ref{BottSam} that $\tilde{\Orb}$ is a (simply connected) homology sphere, hence a homotopy sphere and, by the Poincar\'e Conjecture, a topological sphere. In particular, $\Orb=\tilde{\Orb}/\Gamma$, where $\Gamma=\pi_1^\orb(\Orb)\subset{Iso}(\tilde{\Orb})$.

Furthermore, if $n\neq 3$ and $\Gamma$ contains an involution $g$ acting freely on $\tilde{\Orb}$, then $g$ generates a copy of $\ZZ_2$ acting freely on $\tilde{\Orb}$. By \cite[Theorem 2]{LiSch}, it follows that $\tilde{\Orb}/\langle g \rangle$ (and hence $\tilde{\Orb}$) has constant sectional curvature.

\section{A conjecture}\label{conj}

Recall the construction of weighted projective spaces $\mathbb{CP}_a^n$, the quotients $\Ss^{2n+1}/\Ss^1$ by an almost free action, with $a=(a_0,\ldots,a_n)$ and with $a_i\in \ZZ \backslash\{0\}$ from Section \ref{sub:basics_on_orbifolds}. These constitute simply connected Besse orbifolds in even dimensions. Consequently, our main result cannot be generalized to even dimensions.

However, note that their cohomology was computed in \cite{Hol} to equal
\begin{align*}
H^*_{orb}(\mathbb{CP}_a^n;\Z)=\Z[u] / (a_0\cdot \ldots \cdot a_n \cdot u^{n+1})
\end{align*}
where $u$ generates second cohomology.

As a service to the reader let us quickly recall the argument. We identify $H^*_{orb}(\mathbb{CP}_a^n;\Z)\cong H^*_{\Ss^1} (\Ss^{2n+1};\Z)$ and consider the long exact sequence of the pair $(\mathbb{C}^{n+1}, \Ss^{2n+1})$.
\begin{align*}
\ldots \to H_{\Ss^1}^i(\mathbb{C},\Ss^{2n+1};\Z)\xrightarrow{\alpha} H^i_{\Ss^1}(\mathbb{C}^{n+1};\Z)\to H_{\Ss^1}^i(\Ss^{2n+1};\Z)\to\ldots
\end{align*}
Considering $(\mathbb{C}^{n+1}, \Ss^{2n+1})$ as a pair of trivial vector bundles over the point the Thom isomorphism identifies $H^i_{\Ss^1}(\mathbb{C}^{n+1},\Ss^{2n+1};\Z)\cong H^{i-2(n+1)}_{\Ss^1}(\mathbb{C}^{n+1};\Z)$, and $\alpha$ is multiplication with the equivariant Euler class $e_{\Ss^1}(\mathbb{C}^{n+1})=a_0\cdot \ldots \cdot a_n\cdot u^{n+1}$ (computed using the splitting principle). Consequently, the morphism $\alpha$ is injective, and the long exact sequence above splits to yield the result. Alternatively, the result can be obtained by an application of the Gysin sequence to the fibration
\begin{align*}
\Ss^1\to \Ss^{2n+1}\times E\Ss^1\to B \mathbb{CP}_a^n.
\end{align*}

The same reasoning shows that the integral cohomology ring of an $4n$-dimensional weighted quaternionic projective space is of the form $\Z[u] / (p \cdot u^{n+1})$ with $u$ of degree $4$ and some integer $p$, which, by  \cite[Theorem 7.7, p.~568]{Qui71}, is $1$ if and only if the corresponding action of $\mathrm{SU} (2)$ on $\Ss^{4n-1}$ is free.

\bigskip

Recall further that due to \ref{BottSam} the best known result in the manifold case is a cohomological classification which, in particular, implies that the integral cohomology algebra is generated by one element only.

Taken together, these observations motivate the following conjecture.
\begin{cnj}
The integral cohomology algebra of a simply-connected even-dimensioeven-dimensionalnal Besse orbifold $\mathcal{O}$ of dimension $d$ is generated by one element $u$ of even degree $k$ such that $d=k\cdot n$ and $k\in \{2,4,8\}$, $n\in \mathbb{N}$, such that
\begin{align*}
H^*_{orb}(\mathcal{O};\Z)\cong \Z[u]/ (l \cdot u^{n+1}).
\end{align*}
\end{cnj}
Note that by \cite[Theorem 7.7, p.~568]{Qui71} in this case $\Orb$ has the structure of a manifold if and only if $|l|=1$.

\begin{rem}
The reader may observe that we can support this conjecture further: Indeed, our Morse theoretic arguments (with rational coefficients) together with the corresponding arguments from rational homotopy theory presented in Proposition \ref{ratellprop} show that also in the case when $\dim \Orb$ is even its rational (orbifold) cohomology algebra is a truncated polynomial algebra generated by one element of even degree.
\end{rem}

\begin{appendix}
\section{Manipulations with Hopf algebras}\label{sec:mani_hopf_algbras}

This appendix contains a few elementary lemmas about differential Hopf algebras, needed in Section \ref{S:Zp-cohomology} but too technical for being included there. We refer the reader to Section \ref{SS:Hopf} for relevant definitions and notation.

Recall that, given a differential Hopf algebra $(A,d)$, then its homology $H_*(A,d)$ has again the structure of a Hopf algebra.

\begin{lem}\label{L:primitive}
Let $(A,d)$ be a differential Hopf algebra, and let If $y\in A$ be primitive of odd degree with $dy=:x$. Then $y'=[yx^{p-1}]\in H_*(A,d)$ is primitive as well.
\end{lem}

\begin{proof}
$\Psi(y')=\Psi[yx^{p-1}]=[\Psi(yx^{p-1})]$, so it is enough to prove that $\Psi(yx^{p-1})\sim 1\otimes yx^{p-1}+yx^{p-1}\otimes 1$, where $\sim$ denotes equivalence up to coboundaries. Since $y$ is primitive, so is $x$, and we get:
\begin{align*}
\Psi(yx^{p-1})=& \Psi(y)\Psi(x)^{p-1}\\
=&(1\otimes y+y\otimes 1)(1\otimes x+x\otimes 1)^{p-1}\\
=&(1\otimes y+y\otimes 1)\left(\sum_{k=0}^{p-1}{p-1\choose k}x^k\otimes x^{p-1-k}\right)\\
=&\sum_{k=0}^{p-1}{p-1\choose k}x^k\otimes yx^{p-1-k}+\sum_{k=0}^{p-1}{p-1\choose k}yx^k\otimes x^{p-1-k}.
\end{align*}
Separating the term $(1\otimes yx^{p-1}+yx^{p-1}\otimes 1)$, this becomes
\[
(1\otimes yx^{p-1}+yx^{p-1}\otimes 1)+\sum_{k=1}^{p-1}{p-1\choose k}x^k\otimes yx^{p-1-k}+\sum_{k=0}^{p-2}{p-1\choose k}yx^k\otimes x^{p-1-k}
\]
\[
= (1\otimes yx^{p-1}+yx^{p-1}\otimes 1)+\sum_{k=1}^{p-1}{p-1\choose k}d(yx^{k-1})\otimes yx^{p-1-k}+\sum_{k=0}^{p-2}{p-1\choose k}yx^k\otimes x^{p-1-k}
\]
\[
\sim (1\otimes yx^{p-1}+yx^{p-1}\otimes 1)+\sum_{k=1}^{p-1}{p-1\choose k}yx^{k-1}\otimes d(yx^{p-1-k})+\sum_{k=0}^{p-2}{p-1\choose k}yx^k\otimes x^{p-1-k}
\]
\[
= (1\otimes yx^{p-1}+yx^{p-1}\otimes 1)+\sum_{k=1}^{p-1}{p-1\choose k}yx^{k-1}\otimes x^{p-k}+\sum_{k=0}^{p-2}{p-1\choose k}yx^k\otimes x^{p-1-k}
\]
\[
= 1\otimes yx^{p-1}+yx^{p-1}\otimes 1+\sum_{k=1}^{p-1}\left({p-1\choose k}+{p-1\choose k-1}\right)yx^{k-1}\otimes x^{p-k}
\]
\[
= 1\otimes yx^{p-1}+yx^{p-1}\otimes 1+\sum_{k=1}^{p-1}{p\choose k}yx^{k-1}\otimes x^{p-k}
\]
Since in characteristic $p$ one has ${p\choose k}=0$ for all $k=1,\ldots p-1$, the last term is just
\[
1\otimes yx^{p-1}+yx^{p-1}\otimes 1.
\]
\end{proof}

\begin{lem}\label{L:decomposable-primitive}
Let $A$ be a Hopf algebra, and $\bar{y}\in A^*$ (the dual of $A$). Then $\bar{y}(w)=0$ holds for every $w$ decomposable, if and only if $\bar{y}$ is primitive.
\end{lem}
\begin{proof}
Let $\Psi(\bar{y})=1\otimes \bar{y}+\bar{y}\otimes1+\sum_i y_i'\otimes y_i''$. Then $Y=\sum_i y_i'\otimes y_i''=0$ in $(\bar{A}\otimes \bar{A})^*$ (where $\bar{A}$ is the subspace of elements of positive degree), if and only if $Y(x'\otimes x'')=0$ for all $x'\otimes x''\in \bar{A}\otimes \bar{A}$. But
$$Y(x'\otimes x'')=\Psi(\bar{y})(x'\otimes x'')=\bar{y}(x' x'')=0$$ since $x'x''$ is decomposable, hence the result.
\end{proof}

The next Lemma, applies to the specific case of Browder's spectral sequence $_sEE_r(B_k)$ of biprimitive Hopf algebras (see Section \ref{S:Zp-cohomology}). Notice that, given two triplets of integers $(s,r,k), (s',r',k')$, with $(s,r,k)<(s',r',k')$ in the lexicographic order, it follows that the page ${}_{s'}EE_{r'}(B_{k'})$ occurs after the page $_sEE_r(B_k)$. In this situation, we say that a subalgebra $E'\subseteq {}_{s'}EE_{r'}(B_{k'})$ \emph{lifts} to a subalgebra $E\subset{}_sEE_r(B_k)$, if there is an algebra isomorphism $\phi:E\to E'$ such that every $x\in E$ represents $\phi(x)$.
\begin{lem}\label{L:lift}
Given two triplets of integers $(s,r,k)<(s',r',k')$ in the lexicographic order, the algebra $_{s'}EE_{r'}(B_{k'})$ lifts to a subalgebra $E\subset{}_sEE_r(B_k)$, such that primitive elements of even degree lift to primitive elements.
\end{lem}
\begin{proof}
Let $E_{s',r',k'}(m)\subseteq{}_{s'}EE_{r'}(B_{k'})$ be the subalgebra generated by primitive elements of degree $\leq m$. We will show that for any $m$, there are lifts of $E_{s',r',k'}(m)$ to subalgebras $E_{s,r,k}(m)$ of $_sEE_r(B_k)$, such that $E_{s,r,k}(m)\subseteq E_{s,r,k}(m+1)$ and so that the lifts are compatible with the inclusions.

We will prove this by induction:
\begin{itemize}
\item If $s'>0$, then $_{s'-1}EE_{r'}(B_k)$ is isomorphic to
\[
\bigotimes_{i=1}^I\left(\Lambda(y_i)\otimes \Zp[x_i]/(x_i^p)\right)\otimes \bigotimes_{j=1}^J\left(\Lambda(u_j)\otimes \Zp[v_j]/(v_j^p)\right)\otimes Q
\]
and $_{s'}EE_{r'}(B_k)$ is isomorphic to
\[
\bigotimes_{i=1}^I\Lambda ([y_ix_i^{p-1}])\otimes \bigotimes_{j=1}^J\Lambda([u_jv_j^{p-1}])\otimes Q.
\]
By the explicit shape of these algebras, it is clear how to lift $_{s'}EE_{r'}(B_{k'})$ to $_{s'-1}EE_{r'}(B_k)$, and in particular any $E_{s',r',k'}(m)$. Notice that primitive elements of even degree must live in the $Q$ factor, hence their lift is again primitive.
\item If $s'=0$ and $r'>0$ then, since $E_{0,r',k'}(m)\subseteq{} _{0}EE_{r'}(B_{k'})$ is nonzero only in degrees $\leq d$ for some $d$, then let $s''$ such that $_{s''}EE_{r'-1}(B_{k'})$ is isomorphic to $_0EE_{r'}(B_{k'})$ in degrees $\leq d$. Then, clearly, $E_{0,r',k'}(m)$ lifts to a subalgebra of $_{s''}EE_{r'-1}(B_{k'})$, and, by iterating the point before, it lifts to $_{0}EE_{r'-1}(B_{k'})$.
\item If $s'=0$ and $r'=0$ then, again since $E_{0,0,k'}(m)\subseteq{} _{0}EE_{0}(B_{k'})$ is nonzero only in degrees $\leq d$ for some $d$, then let $r''$ such that $_{0}EE_{r''}(B_{k'-1})$ is isomorphic to $_0EE_{0}(B_{k'})$ in degrees $\leq d$. As before, $E_{0,0,k'}(m)$ lifts to a subalgebra of $_{0}EE_{r''}(B_{k'-1})$ and, by iterating the point before, it lifts to $_{0}EE_{0}(B_{k'-1})$.
\end{itemize}
Iterating these steps, one can lift $E_{s,r,k}(m)$ to any earlier page. Taking the direct limit as $m\to \infty$, one gets the result.
\end{proof}

\section{Wadsley's Theorem for orbifolds}\label{APP:Wadsley}

\subsection{Smooth actions on orbifolds}\label{sub:smooth_orbifold}

In order to prove Proposition \ref{prp:common_period} on the existence of a minimal common period for the geodesics of a Besse orbifold we first need to explain some properties of smooth actions on orbifolds. We assume that the reader is familiar with the notion of a smooth orbifold (see e.g. \cite{MR2973378}). A map between orbifolds is called \emph{smooth} if it locally lifts to smooth maps between manifolds charts \cite[Def.~2.6]{GKRW}. A smooth action of a Lie group $G$ on an orbifold $\Orb$ is a continuous action of $G$ on $\Orb$ that induces a smooth map $G\times \Orb \To \Orb$. If $G$ acts smoothly on $\Orb$ and $x \in \Orb$, then the orbit map $i_x:G/G_x \To \Orb$, $[g] \mapsto gx$, is smooth, a homeomorphism onto its image and all points on $Gx \subset \Orb$ have the same local group \cite[Lem.~2.11]{GKRW}. The latter moreover shows that $G/G_x$ embeds into $\Orb$ as a \emph{full suborbifold} \cite{MR2973378} in the sense that the images of all lifts of $i: G/G_x \To \Orb$ to good manifold charts are invariant, and hence point-wise fixed, under the actions of local groups (see proof of \cite[Lem.~2.11]{GKRW}). In particular, if $\Gamma_x$ is the local group of $\Orb^n$ at $x$, then the tangent space of $\Orb$ at $x$ splits as $\R^{k}/\Gamma_x \times T_x Gx$, where $k=n-\dim(G/G_x)$. This splitting is preserved under the action of $G$ and so the normal bundle $\nu(Gx)$ of $Gx$ in $\Orb$ is a $\R^k/\Gamma_x$-bundle over $G/G_x$.

\subsection{Proof of Wadsley's theorem}\label{sub:wadsley_theorem}

The geodesic flow induces a smooth $\R$-action on the unit sphere bundle $T^1 \Orb$ of a Besse orbifold $\Orb$ which is itself an orbifold. The orbits of this action are geodesics with respect to the natural (Sasaki) metric on $T^1 \Orb$, they project to the geodesics on $\Orb$ with the same arc-length parametrization and the same period, and all the geodesics on $\Orb$ can be obtained in this way (cf. \cite[Ch.~1.K]{Besse}). In the manifold case a theorem by Wadsley says that an $\R$-action with these properties factors through an $\Ss^1$-action \cite{MR0400257}, see also \cite[Appendix A]{Besse}. We are going to argue that the same statement and its conclusion holds in the present setting of orbifolds. In the manifold case the proof involves a local analysis of Poincar\'e return maps. In order to have the same tool available in the orbifold case we first prove the following statement.

\begin{lem}\label{lem:local_development} Let $\varphi: \R \times \Orb^n \To \Orb^n$ be a smooth action on an orbifold $\Orb$ and suppose that $L$ is a periodic orbit. Then there exists a small neighborhood $W$ of $L$ in $\Orb$ which is covered by a manifold $p: M \To W$ such that the restriction of $p$ to the preimage of $L$ in $M$ is a homeomorphism.
\end{lem}
\begin{proof} According to our discussion in Section \ref{sub:smooth_orbifold} the normal bundle $\nu (L)$ of $L$ in $\Orb$ is an $\R^{n-1}/\Gamma_x$-bundle over $\Ss^1$ where $\Gamma_x$ is the local group of a point on $L$. Using the exponential map we find a small neighborhood $W$ of $L$ in $\Orb$ diffeomorphic to $\nu(L)$. We identify $W$ and $\nu (L)$ via this diffeomorphism. Let $\bar W$ be the universal covering of $W$ as a topological space with a smooth orbifold structure induced form $W$ so that $\bar W \To W$is also a covering of orbifolds. Since $W$ is homotopy equivalent to $\Ss^1$ the deck transformation group of the covering $\bar W \To W$ is isomorphic to $\ZZ$ and by construction it acts transitively on the fibers, i.e. the covering is normal. The space $\bar W$ is an $\R^{n-1}/\Gamma_x$-bundle over $\R$ and, in fact, it as a trivial bundle. In this case this can for instance be seen by lifting the induced $\R$-action $d\phi:\R \times \nu (L) \To \nu (L)$ to $\bar{W}$. In particular, its orbifold fundamental group is isomorphic to $\Gamma_x$. Moreover, it shows that the universal (orbifold) covering $\tilde{W}$ of $W$ is $\R^{n-1}\times \R$. Let $\Gamma$ be the deck transformation group of the covering $\tilde W \To W$ and let $N$ be its subgroup isomorphic to $\Gamma_x$ corresponding to the covering $\tilde W \To \bar W$. Since the covering $\bar W \To W$ is normal, the subgroup $N$ is normal in $\Gamma$ and the quotient $\Gamma/N$ is isomorphic to $\ZZ$, the deck transformation group of the covering $\bar W \To W$. This implies that $\Gamma$ is a semidirect product $N \rtimes H$ of $N$ with an infinite cyclic subgroup $H$ of $\Gamma$. Since $H$ does not have finite subgroups, it acts freely on $\tilde W$ and so $M=\tilde W/H$ is a manifold. By construction the group $N$ leaves the preimage of $L$ in $\tilde{W}$ point-wise fixed. It follows that the induced covering $p:M\To W$ restricts to a homeomorphism on its preimage of $L$. This completes the proof of the lemma.
\end{proof}

In the situation of the preceding lemma the vector field on $W$ induced by the $\R$-action lifts to a vector field on $M$ whose local flow covers the flow on $W$. A small disk $D$ transversal to the periodic orbit $L$ in $W$ at a point $x$ is covered by a small disk $\hat D$ in $M$ transversal to the homeomorphic lift $\hat L$ of $L$ in $M$. In $M$ we can find a smaller disk $\hat E\subset \hat D$ centered on $\hat L$ for which a smooth Poincar\'e map $\hat\theta: \hat E \To\hat D$ with smooth return time $\hat\tau: \hat E \To \R$ is defined, cf. \cite[A.14]{Besse}. Setting $E= p(\hat E)$ we then obtain a smooth Poincar\'e map $\theta: E \To D$ covered by $\hat \theta$ whose first return time $\tau: E \To \R$ satisfies $\tau \circ p=\hat \tau$ and so is smooth as well. Let $i>0$ and suppose that $\hat E$ is sufficiently small so that $\hat \theta^{i|\Gamma_x|}:\hat E \To \hat D$ and $\theta^{i}: E \To D$ are defined. Then a preimage in $\hat E$ of a fixed point $y \in E$ of $\theta^i$ is a fixed point of $\hat \theta^{i|\Gamma_x|}$.

Now we explain the proof of the orbifold version of Wadsley's result which implies our Proposition \ref{prp:common_period}.

\begin{prp} Let $\varphi: \R \times \Orb \To \Orb$ be a smooth action on an orbifold such that all orbits are periodic geodesics parametrized by arclength by the action. Then the $\R$-action factors through an $\Ss^1$-action, i.e. all orbits have a minimal common period.
\end{prp}
\begin{proof} We only explain those parts of the proof in detail that require modifications compared to the proof in the manifold case given in \cite{Besse}. Otherwise we refer to \cite{Besse} for the details. We follow the notation from \cite{Besse}. Let $\varrho: \Orb \To \R$ be the function that assigns to a point $x\in \Orb$ the length of the orbit through $x$. The function $\varrho$ is not necessarily locally bounded. However, if the orbit of a point $x_0\in E$ (notation as in the preceding paragraph) defined by $x_i=\theta^i(x_0)$, $i>0$, remains in $E$ for every value of $i$, then we have the following relation. Since the periodic orbit through $x$ meets $E$ only a finite number of times, the $\theta$-orbit of $x$ is finite, say of order $n(x)$, we have
\begin{equation} \label{eqn:A16}
\varrho(x)=\sum^{n(x)-1}_{i=0} \tau(x_i).
\end{equation}

Since the local analysis using Poincar\'e maps is available in the orbifold setting, the following statement can be proven verbatim as in the manifold case, see \cite[Prop.~A.17]{Besse}.
\begin{prp} \label{A.17} Let $L$ be a fixed orbit of length $\lambda$. Let $k>0$ and $\varepsilon>0$ be given. Then there is a neighborhood $U$ of $L$ such that if $x\in U$ and $\varrho (x)\leq k$, then
\begin{compactenum}
\item for some integer $i$ with $0\leq i\leq\lfloor k/\lambda \rfloor+1< \varepsilon$ we have $|\varrho(x)-i\lambda|<\epsilon$ and
\item the orbit through $x$ lies in an $\varepsilon$-neighborhood of $L$.
\end{compactenum}
\end{prp}
Using this proposition the following corollary is proven in the same way as in \cite[Cor.~A.18]{Besse}.
\begin{cor} \label{A.18} The function $\varrho$ is lower semicontinuous. If $N$ is any compact subset of $\Orb$, then the set $V$ of points of continuity of $\varrho_{|N}$ is an open subset of $N$. If $N$ is locally compact, then $V$ is dense in $N$.
\end{cor}
Now suppose that $\varrho$ is bounded by some $k>0$ in a neighborhood of an orbit $L$  through a point $x$. Then as in \cite[A.19]{Besse} one can apply Proposition \ref{A.17} to find a small Poincar\'e disk $E'$ centered at $x$ with $\theta(E')=E'$. Since $\varrho$ is bounded by $k$ on $E'$ the orders of the $\theta$-orbits are bounded by some $m$. For such an $m$ the composition $\theta^{m!}$ is the identity on $E'$. The map $\theta:E'\To E'$ is covered by a diffeomorphism $\hat \theta:\hat E'\To \hat E'$ of finite order of a disk $\hat E'$ in $M$. Using this cover the argument from \cite[A.20]{Besse} shows that the set of points in $E'$ whose $\theta$-orbit has the same order as $\theta$, say $r$, is open and dense in $E'$. Therefore, equation (\ref{eqn:A16}) shows that the function $\alpha(x)=\limsup_{y\rightarrow x}\varrho(y)$ satisfies
\begin{equation} \label{eqn:A22}
\alpha(x)=\sum^{r-1}_{i=0} \tau(\theta^i x)
\end{equation}
for $x\in E'$. Since $\alpha$ is invariant under the flow, this shows that $\alpha$ is smooth in a neighborhood of $L$. Moreover, if $X$ denotes the vector field defined by the $\R$-action on $\Orb$, then the flow of the vector field $X/\alpha$ in a neighborhood of $L$ factors through an $\Ss^1$-action. Hence, in order to prove the proposition, it suffices to show that $\alpha$ is constant on $\Orb$.

Set $B_1=\{x\in \Orb \mid \varrho \text{ is unbounded in every neighborhood of }x\}$ and write $\varphi_t(x)$ for $\varphi(t,x)$. By continuity we have $\varphi_{\alpha(x)}(x)=x$ for all $x\in \Orb \backslash B_1$. Let $\gamma:(-1,1)\To \Orb$ be a smooth path with $\gamma(0)=x$ and consider the one-parameter family of geodesics defined by
\[
		(s,t) \longmapsto \varphi_{t\alpha(\gamma(s))}(\gamma(s))
\]
for $t\in [0,1]$. For fixed $s$ the length of this geodesic is $\alpha(\gamma(s))$. On the other hand, the first variation formula, which also works for orbifolds, implies that this length is independent of $s$. This implies that $\alpha$ is locally constant on $M\backslash B_1$. Hence, in order to prove the proposition, it suffices to show that $B_1$ is empty.

The proof of $B_1 = \emptyset$ is an application of a theorem by Newman and works by contradiction. Since Newman's theorem only applies to manifolds the proof in the orbifold case has to be slightly refined as follows compared to the proof in \cite[p.~220]{Besse}. Let $B_2=\{x\in B_1 \mid \varrho_{|B_1} \text{ is not continuous at } x\}$. Corollary \ref{A.18} implies that $B_1$ is closed and has void interior in $\Orb$ and that $B_2$ is closed with void interior in $B_1$.
Let $x \in B_1 \backslash B_2$ and let $U$ be a connected open neighborhood of $x$ which is a quotient of a manifold chart by the local group $\Gamma_x$, which is disjoint from $B_2$ and such that
\begin{equation}\label{eq:A.27}
|\varrho(u)-\varrho(v)|<\varrho(w)/8 \text{ for } u,v,w \in U \cap B_1.
\end{equation}
Such an $U$ exists since $\varrho$ is lower semicontinuous on $\Orb$ and its restriction to $B_1$ is continuous on $B_1 \backslash B_2$.
Let $V$ be a component of $\Orb\backslash B_1$ which meets $U$ and such that
\begin{equation}\label{eq:A.28}
c>3|\Gamma_x|\varrho(x), \text{ where } \alpha_{|V}\equiv c.
\end{equation}
By continuity the restriction of the time $c$ flow map $\varphi_c$ to $\bar{V}$ is the identity, and by definition of $V$ we have $\partial V \subset B_1$. Moreover, we have
\begin{equation}\label{eq:A.29}
\mathrm{int}(\bar V)=V.
\end{equation}
Since $U$ is connected and meets $B_1$, there is a point $y\in U \cap \partial V$. By standard orbifold theory the order of $\Gamma_y$ divides the order of $\Gamma_x$. Since $\varphi_c(y)=y$, there is an integer $k$ such that $c=k\varrho(y)$. By (\ref{eq:A.27}) and (\ref{eq:A.28}), $k\geq 2 |\Gamma_y|$.

Let $p:M \To W$ be given by Lemma \ref{lem:local_development} applied to the periodic orbit $L$ through $y$. Set $V'=p^{-1}(V\cap W)$ and let $\hat y \in M$ be a preimage of $y$. Let $\hat D$ be a small open disk of codimension one with $p(\hat D)\subset U$, transverse to the orbit through $\hat y$ and with center $\hat y$. Let $\hat E$ be a smaller concentric disk on which $\hat \theta^i$ is defined for $1\leq i\leq k!|\Gamma_y|$, with each $\hat \theta^i \hat E\subset \hat D$. By lower semicontinuity of $\varrho$ and the fact that $\varphi_c$ is the identity on $\bar V$ we can assume that the restriction of $\theta^{k!}$ to $p(\hat E)\cap \bar V$ is the identity. Then also the restriction of $\hat \theta^{k!|\Gamma_y|}$ to $\hat E\cap \bar V'$ is the identity. It follows as in \cite{Besse} that there exists an open, connected and $\Gamma_y$-invariant neighborhood $J$ of $\hat y$ in $E$ with
\[
			\hat \theta (J \cap \bar V ')=J \cap \bar V '.
\]
We define $\psi: J \To J$ by $\psi(z)=\hat \theta^{|\Gamma_y|}(z)$ for $z \in J \cap \bar V' $ and $\psi(z)=z$ otherwise. By (\ref{eq:A.27}) the restriction of $\theta$ to $p(\hat E) \cap \partial V$ is the identity. Hence, also the restriction of $\hat \theta^{|\Gamma_y|}$ to $\hat E\cap \partial V'$ is the identity and so the map $\psi$ is continuous. Since $\psi^{k!}=id$, $\psi$ is a homeomorphism of finite order. Since $k\geq 2 |\Gamma_y|$, $\psi$ is not the identity on $J \cap V'$. By Newman's theorem $J\cap V'$ is dense in $J$. Since $p$ is open, it follows that $V$ is dense in a neighborhood of $y$, and hence $\varphi_c$ is the identity in a neighborhood of $y$. This contradicts the fact that any neighborhood of $y$ contains points from $B_1$, because of $y \in \partial V$ and (\ref{eq:A.29}). The proof is complete.
\end{proof}

\end{appendix}

\end{document}